\documentclass[hidelinks,onefignum,onetabnum]{siamart250211}

\usepackage{mathtools}
\usepackage{algorithm}
\usepackage{graphicx}
\usepackage{float}
\usepackage{booktabs}
\usepackage{subcaption}
\usepackage{tikz}
\usepackage{multirow}
\usepackage{verbatim}
\usepackage{latexsym}
\usepackage{epsfig}
\usepackage{diagbox}
\usepackage{soul}
\usepackage{array}
\usepackage[numbers,sort&compress]{natbib}

\newsiamthm{assumption}{Assumption}
\numberwithin{equation}{section}
\numberwithin{figure}{section}

\usepackage{lipsum}
\usepackage{amsfonts}
\usepackage{graphicx}
\usepackage{epstopdf}
\usepackage[title]{appendix}
\usepackage{algorithm}
\usepackage{algorithmicx}
\usepackage[noend]{algpseudocode}
\usepackage{amsmath}
\usepackage{amssymb}
\usepackage{tabularx}
\usepackage{bm}

\ifpdf
  \DeclareGraphicsExtensions{.eps,.pdf,.png,.jpg}
\else
  \DeclareGraphicsExtensions{.eps}
\fi

\newsiamremark{remark}{Remark}
\newsiamremark{hypothesis}{Hypothesis}
\crefname{hypothesis}{Hypothesis}{Hypotheses}
\newsiamthm{claim}{Claim}
\newsiamremark{fact}{Fact}
\crefname{fact}{Fact}{Facts}

\newsiamremark{example}{Example}
\headers{SLDG on curvilinear mesh}{X. Cai, Y. Chen, K. Fu, and L. Pan}



\title{\small A Third-order Conservative Semi-Lagrangian Discontinuous Galerkin Scheme For the Transport Equation on Curvilinear Unstructured Meshes\footnotemark[1]}
\author{Xiaofeng Cai\footnotemark[2] \footnotemark[3] \and Yibing Chen\footnotemark[4] \and Kunkai Fu\footnotemark[3] \footnotemark[5] \and Liujun Pan\footnotemark[6]}


\usepackage{amsopn}

\begin{document}
	
	\maketitle
	
	\footnotetext[1]{\funding{The work of the first author and the third author were partially supported by the NSFC (No 12201052), National Key Laboratory for Computational Physics (6142A05230201), the Guangdong Provincial Key Laboratory of IRADS (2022B1212010006), Guangdong basic and applied basic research foundation [2025A1515012182], Guangdong and Hong Kong Universities "1+1+1" Joint Research Collaboration Scheme [2025A0505000014].}}
	\footnotetext[2]{Department of Mathematics, Faculty of Arts and Sciences, Beijing Normal University, Zhuhai, 519087, P.R. China. (\email{xfcai@bnu.edu.cn}).}
	\footnotetext[3]{Guangdong Provincial/Zhuhai Key Laboratory of Interdisciplinary Research and Application for Data Science, Beijing Normal-Hong Kong Baptist University, Zhuhai, 519087, P.R. China.}
	\footnotetext[4]{Institute of Applied Physics and Computational Mathematics and National Key laboratory of Computational physics, and Center for Applied Physics and Technology, Peking University, P.O. Box 8009, Beijing 100088, P.R. China.}
	\footnotetext[5]{Hong Kong Baptist University, Kowloon Tong, Hong Kong. (\email{fukunkai@gmail.com}).}
   \footnotetext[6]{Institute of Applied Physics and Computational Mathematics and National Key laboratory of Computational physics, P.O. Box 8009, Beijing 100088, P.R. China.}
	% REQUIRED
	\begin{abstract}
		We develop a third-order conservative semi-Lagrangian discontinuous Galerkin (SLDG) scheme for solving linear transport equations on curvilinear unstructured triangular meshes, tailored for complex geometries. To ensure third-order spatial accuracy while strictly preserving mass, we develop a high-order conservative intersection-based remapping algorithm for curvilinear unstructured meshes, which enables accurate and conservative data transfer between distinct curvilinear meshes.
		Incorporating this algorithm, we construct a non-splitting high-order SLDG method equipped with weighted essentially non-oscillatory  and positivity-preserving limiters to effectively suppress numerical oscillations and maintain solution positivity. 
		For the linear problem, the semi-Lagrangian update enables large time stepping, yielding an explicit and efficient implementation.
		Rigorous numerical analysis confirms that our scheme achieves third-order accuracy in both space and time, as validated by consistent error analysis in terms of $L^1$ and $L^2$-norms.
		Numerical benchmarks, including rigid body rotation and swirling deformation flows with smooth and discontinuous initial conditions, validate the scheme's accuracy, stability, and robustness.
	\end{abstract}
	
	% REQUIRED
	
	% REQUIRED
	
	\section{Introduction}

	Transport processes pervade computational physics, ranging from fluid dynamics \cite{arbogast2011stability,ma2022fourth}  to atmospheric modeling \cite{lauritzen2010conservative,nair2002efficient}. Researchers have make many attempt to solve the transport equation.
	Classical schemes include the upwind method
	\cite{courant1952solution,ocher1982upwind} and Lax-Friedrichs scheme \cite{lax1954weak}. MUSCL scheme \cite{van1979towards} extend Godunov scheme \cite{godunov1959finite} to second-order. The total variation diminishing scheme \cite{harten1983high} start the age of high-resolution. The essentially non-oscillatory and weighted essentially non-oscillatory (WENO) schemes \cite{harten1997uniformly,liu1994weighted} have been designed for maintaining both uniform high order accuracy and an essentially non-oscillatory shock transition. The Runge-Kutta (RK) discontinuous Galerkin (DG) method \cite{cockburn2001runge} effectively leverages the strengths of DG framework, owning compactness, high resolution \cite{reinarz2020exahype}, parallel efficiency, and flexibility for dealing with unstructured meshes \cite{chen2020review}. 
	
	Despite these methods, Semi-Lagrangian (SL) schemes \cite{ewing2001summary,morton1998analysis,sun2025fourth} can overcome the restriction of a Courant-Friedrichs-Lewy (CFL) stability condition. SL methods had been applied to numerous areas, from environmental engineering \cite{kumar1995semi,yearsley2009semi} and image processing \cite{carlini2007semi} to plasma physics \cite{besse2003semi,grandgirard20165d}. However, a well-known drawback of SL methods is that they do not inherently ensure local mass conservation \cite{courant1952solution}. This shortcoming is particularly problematic for complex systems such as multi-tracer transport in chemistry–climate models, where strict conservation of each tracer is required. To address this issue, researchers have developed flux-form SL schemes \cite{dukowicz2000incremental,lin1996multidimensional,restelli2006semi} that  approximating the reconstruction of the flux through the control volume boundary, another strategy is based on conservative remapping \cite{hirt1974arbitrary}, which is discretized by approximating integral over upstream control volume \cite{laprise1995class,machenhauer1998design,zerroukat2002slice}. Combining the DG approach \cite{cockburn2012discontinuous,tumolo2013semi} with SL techniques yields the semi-Lagrangian DG (SLDG) method \cite{restelli2006semi,qiu2011positivity,rossmanith2011positivity} to maintain the local mass conservation. The SLDG method in \cite{guo2014conservative} is based on the weak characteristic Galerkin formulations in \cite{russell2002overview}. In particular, the SLDG schemes use the trace of the characteristics backward in time and well approximation for the upstream cell; such schemes in \cite{cai2017high,cai2022eulerian} can allow for huge time-stepping size. Besides, the potential loss of accuracy of traditional SL methods at low CFL numbers \cite{falcone1998convergence} disapears in SLDG schemes. Early multi-dimensional SLDG solvers often relied on dimensional splitting \cite{qiu2011positivity}, which is straightforward but introduces splitting errors and limited to Cartesian topologies.
	Subsequent developments moved toward non-splitting formulations that evolve multidimensional upstream elements via characteristic weak forms \cite{celia1990eulerian} to avoid such errors~\cite{cai2017high,restelli2006semi} and also suitable for fully unstructured meshes \cite{cai2022eulerian}. 
	
	Advances in mesh generation techniques have enabled the widespread use of unstructured meshes in computational fluid dynamics, computer-aided design, and computer graphics. Unstructured meshes can conform to complex geometries and topologies, accurately representing domains with curved surfaces or patches. They also allow fine local control of mesh size and node density, enabling high resolution of localized features such as discontinuities. And curvilinear unstructured meshes can provide extra precise approximation of complex geometries and topologies. However, the complex structure of curvilinear unstructured meshes increases the computational cost of certain operations (such as intersection searching) and poses challenges for designing high-order conservative remapping algorithms. In \cite{cai2022eulerian}, a second-order conservative SLDG scheme on unstructured meshes was proposed, which transfer solution between triangles, and the major challenge in achieving higher-order accuracy is the lack of a high-order conservative remapping algorithm on curvilinear unstructured meshes, the comparison of them is provided in Section~\ref{section:numerical}.
	
	As an important step in the SLDG transport solver, a robust remapping method must be high-order, conservative and preserve accuracy of discontinuities and positivity for complex domain. Compare with advection-based remapping method \cite{anderson2015monotonicity,lipnikov2020conservative} that solves a linear transport equation over a pseudo-time interval and flux-based remapping method \cite{boutin2011extension,dukowicz1984conservative,vaccaro2023applying} that exchanges flux that is evaluated by integrating solution over a swept region, intersection-based remapping method \cite{dukowicz1987accurate} computes the integral on overlapping area between meshes with different topological connectivities and free from CFL constraint, therefore owning more robustness and efficiency \cite{lipnikov2023conservative}. Recently, there are many research and progress in constructing intersection-based remapping algorithms for straight-side meshes \cite{kenamond2021positivity,kenamond2021intersection}. For 2D Lagrangian schemes with straight-side meshes, it can be at most second-order accurate \cite{cheng2007high}, which necessitates the development of curvilinear remapping method applied in curvilinear Lagrangian schemes and ALE methods. The implementation of finding accurate curvilinear intersection and update solution under DG framework with curvilinear elements poses a non-trivial problem for geometric complexity (such as curvilinear concave region) and challenge of interpolating high-order piecewise polynomials over curvilinear elements. In \cite{shashkov2024remapping}, a strategy is provided that divides complex cell into convex sub-cells and applies the allocation law of sets to identify intersections to achieve accurate remapping between quadratic curvilinear elements. Recently, a higher-order remapping method \cite{lei2025high} is developed under FVM framework. However, analogous methods for curvilinear unstructured meshes and DG framework have yet to be developed. 
	
	In this paper, we present a two-dimensional (2D) unsplit SLDG scheme for transport simulation that incorporates a novel intersection-based remapping method for curvilinear unstructured meshes.
	This approach is the first to achieve third-order accuracy for remapping between unstructured meshes with curved elements. 
	The proposed scheme is also the first third-order (in space and time) conservative method for curvilinear unstructured meshes that permits huge time steps without compromising stability. 
	To suppress spurious oscillations and preserve positivity, we use weighted essentially non-oscillatory (WENO) reconstructions \cite{zhu2013runge} together with a positivity-preserving (PP) limiter \cite{zhang2011maximum}. For the one-dimensional case with a constant coefficient, an optimal superconvergence result has been proven \cite{yang2020optimal}, but a rigorous two-dimensional analysis remains open. We provide a theoretical proof of consistency. We also verify the scheme's stability through numerical experiments and demonstrate that it attains third-order accuracy in both space and time. 
	
	The remainder of the paper is organized as follows. Section~\ref{section:sldg} presents the SLDG formulation. Section~\ref{section:remap} details the remapping algorithm. Section~\ref{section:analytical} presents theoretical properties of the proposed scheme and provides rigorous proof. Section~\ref{section:numerical} reports numerical results, followed by conclusions in Section~\ref{section:conclusion}.

\section{SLDG Scheme for Transport Equation} \label{section:sldg}
This section develops a 2D SLDG method on unstructured meshes. 
%In contrast to~\cite{cai2022eulerian}, we approximate upstream elements by \emph{quadratic–curvilinear} triangles, which enables third-order accuracy. 
%A rigorous consistency proof is provided at the end of the section.

\subsection{The SLDG formulation on Unstructured Meshes}
We consider the 2D linear transport equation,
	$$ 
		\frac{\partial u}{\partial t}+ \nabla_{x,y} \cdot ( \boldsymbol{V}(x, y, t) u )=0,
 \ \ (x,y)\in \Omega, 
	$$
with appropriate boundary conditions and the initial condition $u(x, y, 0)=u_0(x,y)$, where the velocity field $\boldsymbol{V}(x, y, t) = (a(x, y, t), b(x, y, t))$.

%	A set of triangular elements $K_j,\ \ j=1,...,J$ forms a fixed background grid, which is a triangulation of $\Omega$.
% Let $h=\sup\limits_{j=1}^Jd(K_j)$, where $d(K_j)$ denotes the diameter of $K_j$.
%The discontinuous Galerkin finite element space is defined as $V_h^k=\{v_h: v_h|_{K_j}\in P^k(K_j)\}$, where $P^k(K_j)$ is the space of polynomials of degree at most $k$ on $K_j$.
% In particular, let $n_k$ denote the dimension of $P^k(K_j)$, given by $n_k=\frac{(k+1)(k+2)}{2}$.
%Throughout, we adopt a set of local orthogonal basis functions on the triangular element $K$, as defined   in \cite{zhu2013runge}.
	
 Let $\{K_j\}_{j=1}^J$ be a triangulation of $\Omega$. Set $h=\max_{1\le j\le J}\mathrm{diam}(K_j)$, where $\mathrm{diam}(K_j)$ denotes the diameter of $K_j$. Define
$
V_h^k:=\{\,v_h:\ v_h|_{K_j}\in P^k(K_j)\,\},
$
where $P^k(K_j)$ denotes polynomials on $K_j$ of total degree $\le k$. The local dimension is $n_k=\tfrac{(k+1)(k+2)}{2}$. We use local orthogonal bases on each triangle $K$ as in~\cite{zhu2013runge}.

% To construct the SLDG scheme, we consider the function $\psi(x,y,t)$,  which satisfies the adjoint  problem \cite{celia1990eulerian,guo2014conservative} of the transport equation:
%	\begin{equation}\label{adjoint}
%		\left\{\begin{array}{l}\psi_t+a(x,y,t) \psi_x+b(x,y,t) \psi_y=0, \\ \psi\left(x,y,t^{n+1}\right)=\Psi(x,y),\ \ t \in\left[t^n, t^{n+1}\right],\end{array}\right.
%	\end{equation}
%	where the final value of the adjoint problem located in polynomial space, that is $\Psi(x,y)\in P^k(K_j)$. 
%	Along the characteristic trajectories $\gamma:=
%	\left\{
%	\begin{array}{rl}
%		\frac{dx}{dt}&=a(x,y,t)\\
%		\frac{dy}{dt}&=b(x,y,t)
%	\end{array}
%	\right.$, the total derivative of $\psi$ is
%	\begin{equation}\label{prop_adjoint}
%		\frac{d\psi}{dt}=\frac{\partial \psi}{\partial t}+\frac{\partial \psi}{\partial x}\frac{dx}{dt}+\frac{\partial \psi}{\partial y}\frac{dy}{dt}=\psi_t+a(x,y,t) \psi_x+b(x,y,t) \psi_y=0.
%	\end{equation}
%Thus, the test function $\psi$ remains constant along each characteristic curve.
%	\begin{figure}[H]
%		\centering
%		\includegraphics*[width=0.45\linewidth]{Figures/upstream_trajectory1}
%		\hspace{0.3in}
%		\includegraphics*[width=0.36\linewidth]{Figures/characteristic_curve3.png}
%		\caption{Left panel shows an illustration of the space-time region $\widehat{K}_j(t)\times[t^n,t^{n+1}]$. Right panel shows nodes on Eulerian element $K_j$ and its upstream element $K_j^\star$.}
%		\label{unstructured triangle}
%	\end{figure}

To construct the SLDG formulation, introduce $\psi(x,y,t)$ as the solution of the adjoint problem~\cite{celia1990eulerian,guo2014conservative}:
\begin{equation}\label{adjoint}
\left\{
\begin{array}{l}
\psi_t+a(x,y,t)\psi_x+b(x,y,t)\psi_y=0,\qquad t\in[t^n,t^{n+1}],\\
\psi(x,y,t^{n+1})=\Psi(x,y)\in P^k(K_j).
\end{array}
\right.
\end{equation}
 Along characteristics, $\psi$ is constant.
By the Reynolds transport theorem, we have
	$$ \frac{d}{d t} \int_{\widetilde{K}_j(t)} u(x,y,t) \psi(x,y,t) dV =0,
	$$
	where $\widetilde{K}_j(t)$ is the dynamic region obtained by tracing the characteristics backward from the Eulerian element $K_j$ at $t^{n+1}$ along the characteristics.

\begin{figure}[ht]
  \centering
  \includegraphics[width=0.31\linewidth]{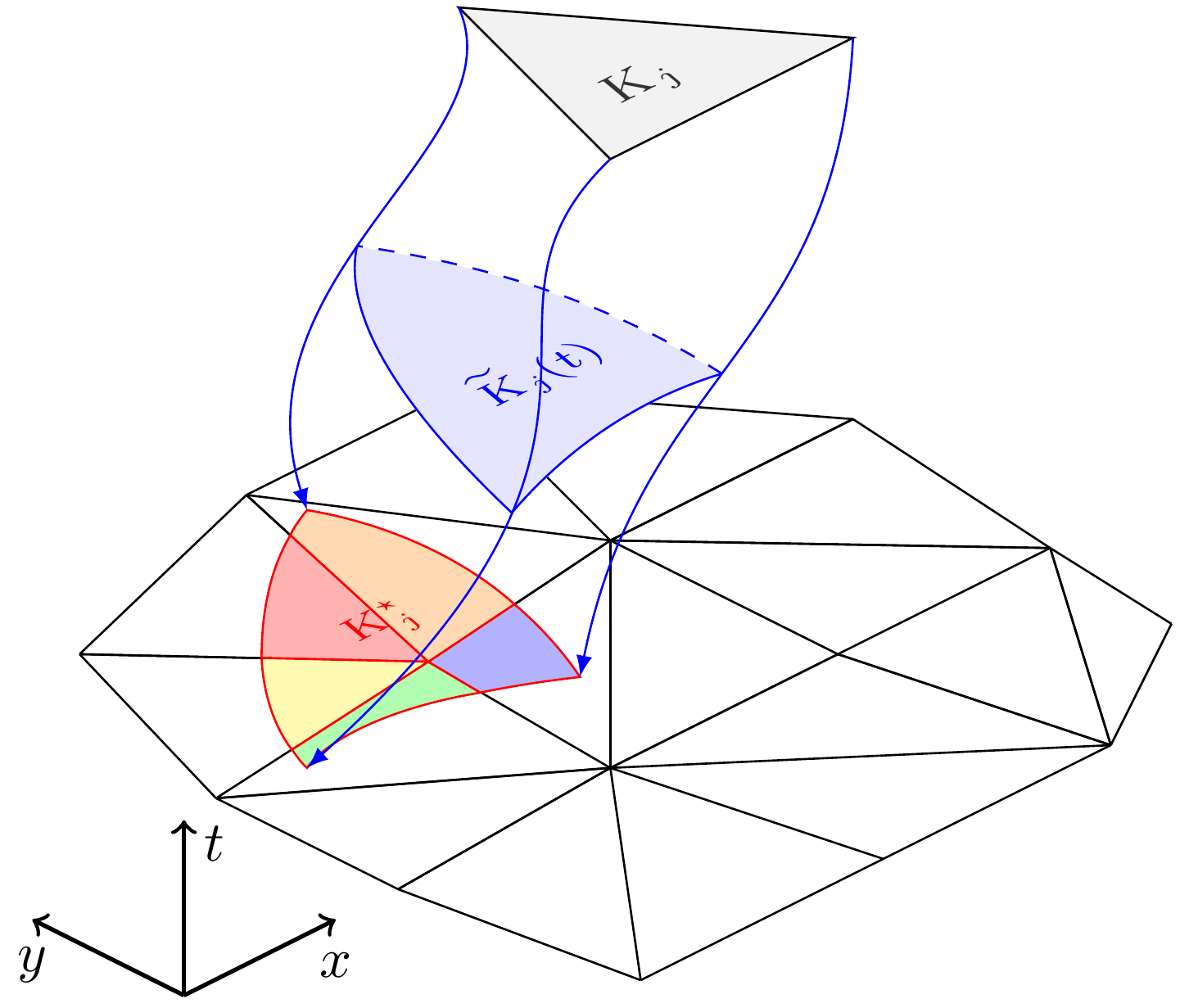}\hspace{0.3in}
  \includegraphics[width=0.26\linewidth]{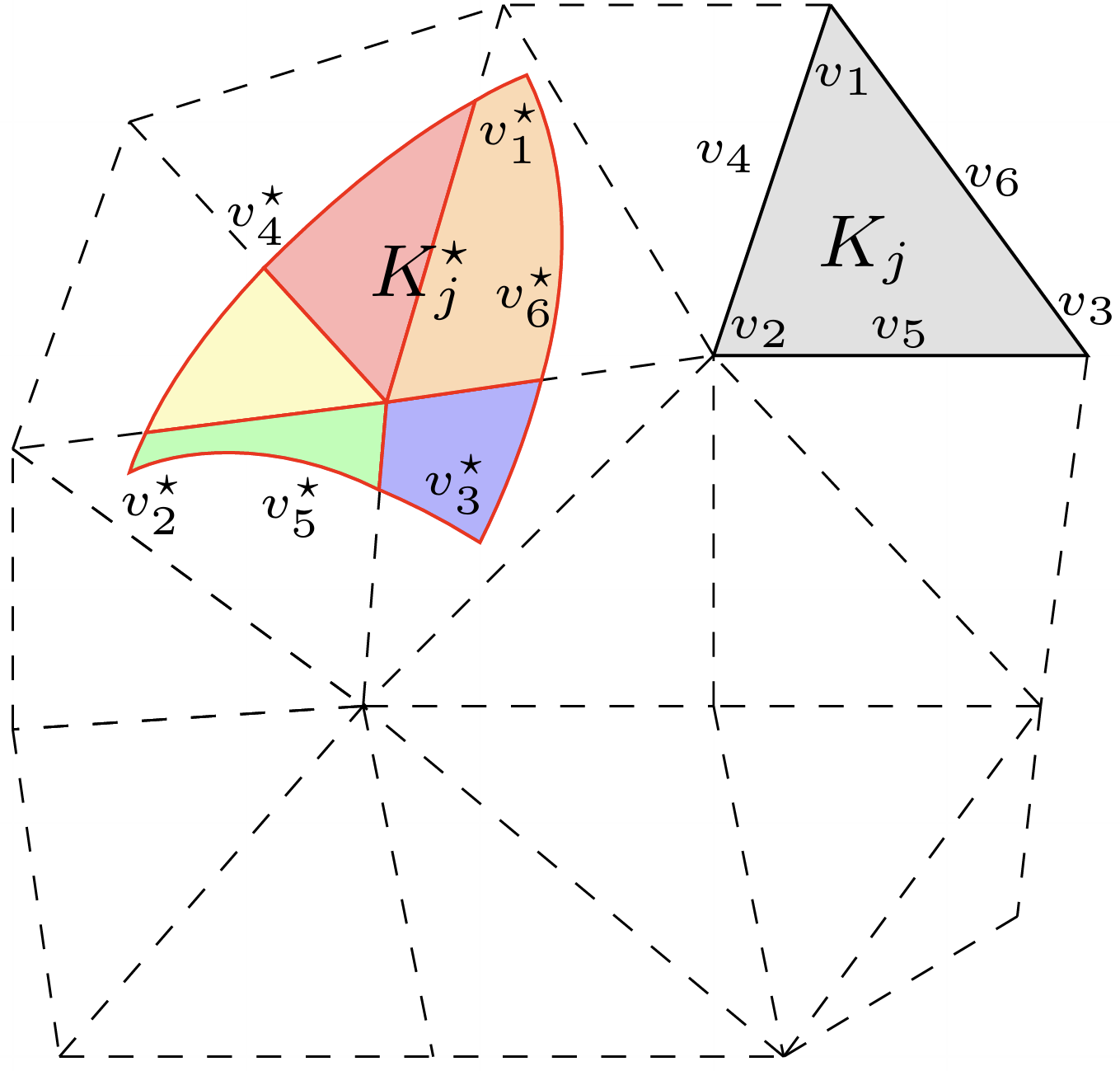}
  \includegraphics*[width=0.35\linewidth]{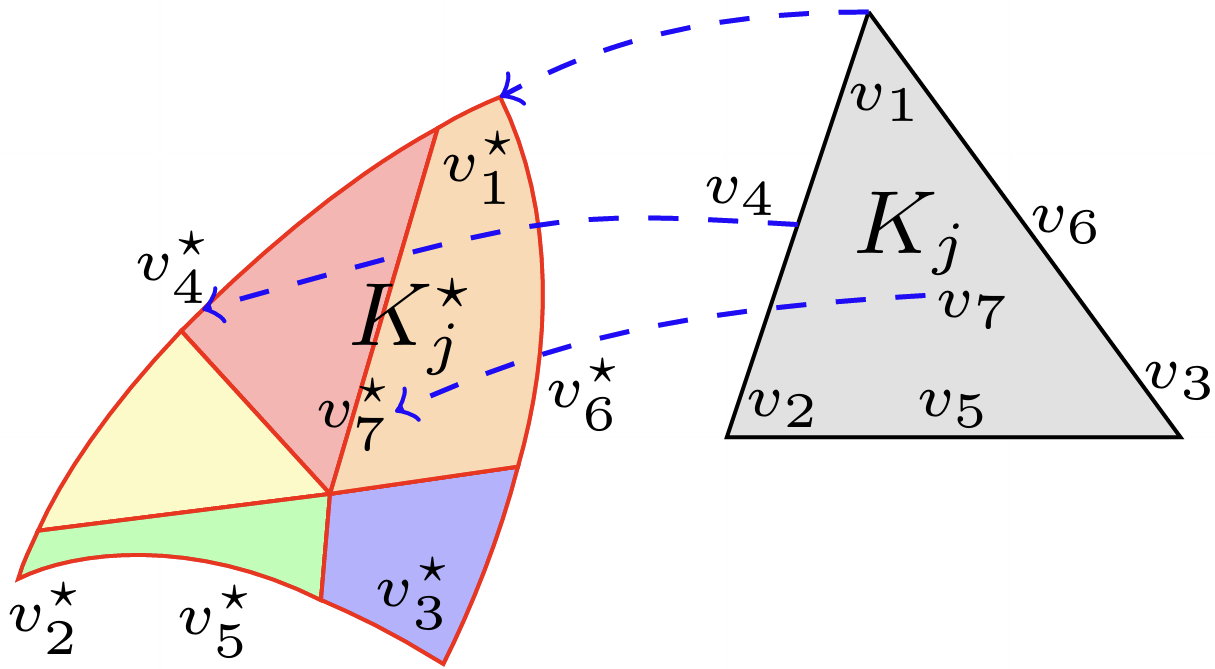}
  \caption{Left: space–time region $\widetilde{K}_j(t)\times[t^n,t^{n+1}]$. Middle: nodes on Eulerian element $K_j$ and its upstream element $K_j^\star$.
  Right: The points $v_q^\star$ on upstream element $K_j^\star$ for least squares approximation and their corresponding points $v_q$ on Eulerian element $K_j$ to evaluate the value of test function $\psi(x,y,t^n)$ on $v_q^\star$, $q=1,2,...,7$.
  }
  \label{unstructured triangle}
\end{figure}

The non-splitting SLDG scheme is then formulated as follows: given $u^n\in V_h^k$ at time level $t^n$, find $u^{n+1}|K_j\in V_h^k$ such that for $\forall\Psi\in P^k(K_j)$, we obtain
	\begin{equation}\label{sl}
		\int_{K_j} u^{n+1}\Psi(x,y)dxdy = \int_{K_j^\star} u^n\psi(x,y,t^n)dxdy,
	\end{equation}
where $K_j^\star$  represents the upstream element of $K_j$ traced backward from  $t^{n+1}$ to $t^n$.

\subsection{Implementation of the 2D SLDG on Unstructured Meshes}\label{procedure_sldg}

%A key task is the accurate and conservative evaluation of the right-hand side of \eqref{sl}, i.e., the integral over the upstream element $K_j^\star$ at time $t^n$. The procedure consists of four steps.

Key to the method is accurate, conservative evaluation of the right-hand side of \eqref{sl}: the integral over the upstream element $K_j^\star$ at $t^n$. The procedure has four components.

%	A key aspect of the SL scheme is the evaluation of the right-hand side (RHS) of equation \eqref{sl}, which involves three main components: 
%\begin{enumerate}
%\item[\romannumeral1.] Approximating the upstream region using a quadratic-curvilinear triangle to achieve third-order accuracy.
%\item[\romannumeral2.] Since  $u^n$ is discontinuous across the interfaces of the background mesh, the integral must be evaluated by considering each subregion separately.
%\item[\romannumeral3.] Reconstructing $\psi(x, y, t^n)$ using the adjoint property ~\eqref{prop_adjoint}.
%\end{enumerate}

%The steps of the SLDG algorithm are as follows:
	\begin{enumerate}
		\item \textbf{Characteristic tracing.} Locate the upstream nodes $v_i^\star, i=1,...,6$ by tracing the characteristics backward to time $t^n$:
		\begin{equation}\label{characteristic}
			\left\{
			\begin{array}{rl}
				&\frac{dx(t)}{dt}=a(x(t),y(t),t), \\
				&\frac{dy(t)}{dt}=b(x(t),y(t),t),\\
				&x(t^{n+1})=x(v_i),y(t^{n+1})=y(v_i),
			\end{array}
			\right.
		\end{equation}
		from the six nodes of $K_j: v_i, i=1,...,6$ backwards shown in Fig.\ref{unstructured triangle} with a high order RK method.
%		\begin{figure}[H]
%			\centering
%			\includegraphics*[width=0.3\linewidth]{Figures/charateristic_curve2.png}
%			\caption{The reference elements (left panel) and material elements (right panel) of a triangular mesh, as well as their corresponding relationships with control points.}
%			\label{characteritics1}
%		\end{figure}

	\item \textbf{Reconstruction of upstream element.} Construct a quadratic curve to approximate sides of upstream element by using a  quadratic isoparametric element TRIA6 \cite{dohrmann2004stabilized} with six control points, which is defined using six nodes $(x_i,y_i), i=1,...,6$ in the physical space and six shape functions $N_i(\xi,\eta)$.

		\begin{enumerate}
			\item Construct the shape function on reference element. To maintain a quadratic coordinate transformation at the boundary of the cell, the shape functions space is spanned by $\{1,\xi,\eta,\xi\eta,\xi^2,\eta^2\}$. The nodal values of the shape functions satisfy:
			\begin{equation*}\label{shapefunction}
				N_i(\xi_j,\eta_j)=a_i + b_i\xi_j + c_i \eta_j + d_i \xi_j\eta_j + e_i\xi_j^2 + f_i\eta_j^2 = \delta_{ij},\ \ i,j\in\{1,...,6\},
			\end{equation*}
			in which Kronecker symbol $
			\delta_{ij}=\left\{
			\begin{array}{l}
				1,\ \ i= j ,\\
				0,\ \ i\neq j.
			\end{array}
			\right.
			$\\
			So we have the shape functions of the TRIA6 element:
			$	N_1(\xi,\eta)=-3\xi+2\xi^2-3\eta+2\eta^2+4\xi\eta+1,  
				N_2(\xi,\eta)=-\xi+2\xi^2,  
				N_3(\xi,\eta)=-\eta+2\eta^2,  
				N_4(\xi,\eta)=4\xi-4\xi\eta-4\xi^2,  
				N_5(\xi,\eta)=4\xi\eta,  
				N_6(\xi,\eta)=4\eta-4\xi\eta-4\eta^2.$
			Note that the aggregate of all basis functions is unity, i.e., $\sum\limits_{i=1}^6N_i(\xi,\eta)=1$. As a consequence, the entire reference space is mapped in such a way that each vertex corresponds to the corresponding vertex in the physical space, as shown in in Fig.\ref{triangle}.
			
		\begin{figure}[ht]
			\centering
			\includegraphics*[width=0.54\linewidth]{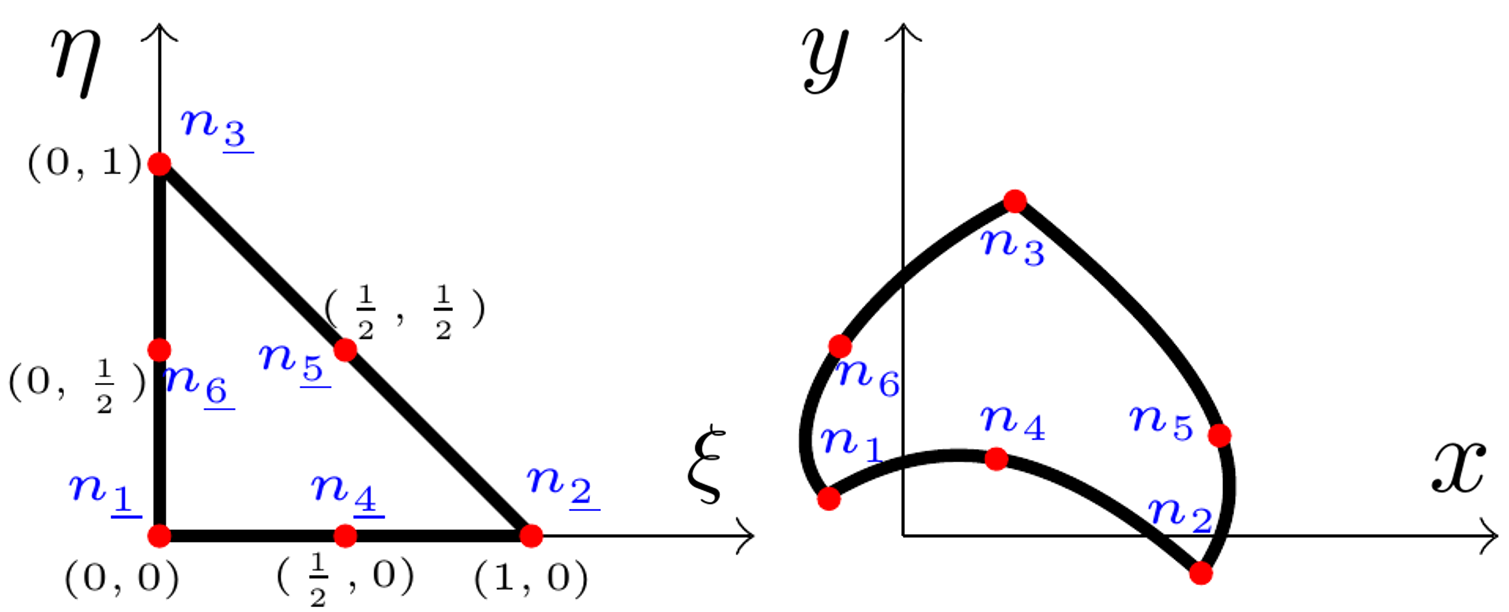}
			\caption{The reference elements (left panel) and material elements (right panel) of a triangular mesh, as well as their corresponding relationships with control points.}
			\label{triangle}
		\end{figure}

\item Obtain the parameter equations for each side of quadratic-curvilinear triangular mesh. The isoparametric map is defined as
			\begin{equation}\label{shapefunction2} \mathbf{x}(\xi,\eta)=\left(x\left(\xi,\eta\right),y\left(\xi,\eta\right)\right)=\left(\sum\limits_{i=1}^6x_iN_i(\xi,\eta),\sum\limits_{i=1}^6y_iN_i(\xi,\eta)\right).
			\end{equation}
			It should be noted that $x\left(\xi,\eta\right),y\left(\xi,\eta\right)$ in \eqref{shapefunction2} can be expressed as bi-quadratic polynomials of $\xi$ and $\eta$. 
			%Take the hypotenuse of the triangle as an example, we substitute $\xi=1-\eta$ into \eqref{shapefunction2}, which ensures that the parameter $\eta$ changes from 0 to 1 along the induced orientation of the curvilinear triangle boundary, so we have the parametric equation
			%$$	\left\{  	\begin{array}{l}
		%	x^H(\eta)=\left(2x_2+2x_3-4x_5\right)\eta^2+\left(-3x_2-x_3+4x_5\right)\eta+x_2,\\
		%		y^H(\eta)=\left(2y_2+2y_3-4y_5\right)\eta^2+\left(-3y_2-y_3+4y_5\right)\eta+y_2.
		%	\end{array}  
		%	\right.$$
			Therefore, we can simplify the parametric equations to 
			\begin{equation}\label{parametric}
				\left[\begin{matrix}
					x^F(\theta) \\
					y^F(\theta) \\
				\end{matrix} \right]=\left[\begin{matrix}
					\alpha^{x,F}_{\theta\theta} \\
					\alpha^{y,F}_{\theta\theta} \\
				\end{matrix} \right]\theta^2+\left[\begin{matrix}
					\alpha^{x,F}_{\theta} \\
					\alpha^{y,F}_{\theta} \\
				\end{matrix} \right]\theta+\left[\begin{matrix}
					\alpha^{x,F} \\
					\alpha^{y,F} \\
				\end{matrix} \right],
			\end{equation}
			where $F=\{L,B,H\}$ denotes any side of the reference triangle, $\alpha$ is a set of coefficients depending on sides and $\theta$ is a parameter, $\theta=\xi, \eta$. Each side of the curvilinear triangle is characterized by a quadratic vector function with respect to $\theta$ and $\theta\in[0,1]$.
		\end{enumerate}

\item \textbf{Least-squares reconstruction of $\psi(\cdot,t^n)$.} 
%		\begin{figure}[ht]
%			\centering
%			\includegraphics*[width=0.35\linewidth]{Figures/characteristics_least_square.png}
%			\caption{The points $v_q^\star$ on upstream element $K_j^\star$ for least squares approximation and their corresponding points $v_q$ on Eulerian element $K_j$ to evaluate the value of test function $\psi(x,y,t^n)$ on $v_q^\star$, $q=1,2,...,7$.}
%			\label{characteristics_least_square}
%		\end{figure}
		Since $\psi$ is constant along characteristics, we have 
		$$
		\psi\left(x\left(v_q^\star\right),y\left(v_q^\star\right),t^n\right)=\Psi\left(x\left(v_q\right),y\left(v_q\right)\right),\quad q=1,2,...,7,
		$$
		where $v_q$ and $v_q^\star$ are illustrated in the right panel of Fig.\ref{unstructured triangle} and $v_7$ denotes the barycenter of element $K_j$. 
	We then reconstruct $\psi^\star\in P^2(K_j^\star)$ by a least-squares fit to $\psi(\cdot,\cdot,t^n)$.

\item \textbf{SLDG remapping algorithm.}
%		\begin{figure}[H]
%			\centering
%			\includegraphics*[width=0.3\linewidth]{Figures/curved_triangle_on_eulerian}
%			\caption{The upstream triangular element (colorful part) on the background mesh.}
%			\label{curved_triangle_on_eulerian}
%		\end{figure}
		Given that the upstream element $K_j^\star$ overlaps with $K_l$ in the background mesh, overlapping subregions, referred to as $K_{j,l}^\star$, can be determined. These subregions are displayed in distinct colors in the right panel of Fig.\ref{unstructured triangle}. Following this, the integral \eqref{sl} can be evaluated for each of these subregions individually based on a clipping algorithm, which will be introduced detailed in section \ref{section:remap}.\\
		To evaluate the integral on subregions, we can introduce two auxiliary function $P(x,y)$ and $Q(x,y)$ such that 
		$$
		-\frac{\partial P}{\partial y}+\frac{\partial Q}{\partial x}=u(x,y,t^n)\psi^\star(x,y),
		$$
		so that the integral on subregions $\int_{K_{j,l}^\star}u(x,y,t^n)\psi^\star(x,y)\mathrm{d}x\mathrm{d}y$ can be converted into line integral via Green's Theorem, i.e.,
		\begin{equation}\label{green_theorem}
			\int_{K_{j,l}^\star}u(x,y,t^n)\psi^\star(x,y)\mathrm{d}x\mathrm{d}y=\oint_{\partial K_{j,l}^\star}P\mathrm{d}x+Q\mathrm{d}y.
		\end{equation}
	%	Note that the choices of auxiliary function $P(x,y)$ and $Q(x,y)$ is not unique. 
%, but the value of integral on subregions is independent of the choices. 

%In the implementation, we collect terms of same order $x^iy^j$ in the integrand of the left-hand side of \eqref{green_theorem} and yields
	%	\begin{equation}\label{green_theorem2}
%			u(x,y,t^n)\psi^\star(x,y)=\sum\limits_{i+j\leqslant2} c^{(i,j)}x^iy^j.
%		\end{equation}

		%The integral of these terms can be derived by using \eqref{green_theorem} with the following $P(x,y)$ and $Q(x,y)$,
%		$$
%		\begin{array}{rll}
%			1:\qquad&P=0,&Q=x,\\ 
%			x:\qquad&P=0,&Q=\frac{x^2}{2},\\
%			y:\qquad&P=-\frac{y^2}{2},&Q=0,\\
%			x^2:\qquad&P=0,&Q=\frac{x^3}{3},\\
%			xy:\qquad&P=0,&Q=\frac{x^2y}{2},\\
%			y^2:\qquad&P=-\frac{y^3}{3},&Q=0.
%		\end{array}
%		$$
	\end{enumerate}

		\section{An Intersection-Based Remapping Algorithm} \label{section:remap}
%	From the description in \textbf{Section 1}, we know that the intersection-based remapping algorithm is the most robust remapping method and can be applied to calculate the integral on upstream elements in SLDG method and transfer data in high-precision indirect arbitrary Lagrangian Eulerian method. The clip algorithm between polygons has been given, however, that between polygon and curved polygon and between curved polygons provide higher precision. And unstructured meshes are appealing for meshing complex geometries, so we use unstructured meshes for the following statements.

This section develops a conservative intersection-based remapping procedure tailored to the SLDG update. Given the upstream element $K_j^\star$
 traced from $K_j$
 at $t^{n+1}$
 to $t^n$
, we compute its overlaps with the background mesh and evaluate the cellwise integrals by decomposing 
$K_j^\star\cap K_l$
 into curved subcells and converting each area integral into an oriented line integral along 
 $\partial(K_j^\star\cap K_l)$
 via Green’s theorem. This construction naturally handles straight-sided and quadratic-curvilinear edges and is therefore well suited to complex geometries; compared with polygon-only clipping, treating intersections that involve curved boundaries yields higher geometric fidelity. The algorithmic core consists of (i) robust subcell construction by clipping $K_j^\star$
against the Eulerian mesh, (ii) line–integral evaluation on each curved subcell boundary, and (iii) a set-theoretic decomposition that reduces general intersections to combinations of convex pieces. The subsequent subsections detail these components and the associated moment formulas. 

	\subsection{Intersection of Polynomial Boundary Elements}
	Unlike straight sided polygons \cite{cheng2008high,kaazempur2002efficient, sutherland1974reentrant}, which we can partition concave polygons into multiple convex polygons, For concave curved polygons, their curvature is negative on all non-zero measurement edges, which makes it impractical to perform a segmentation at every point with negative curvature. Therefore, we need to find a shape outside the shape, for which the curvature on this edge is positive, that is, a convex shape. We first consider the constructed shape and the original shape as a whole, and then cut them out, so that we can turn a concave shape into two convex shapes, where we mark the parts that need to be cut as $"-"$ and the rest as $"+"$.

	\begin{figure}[htbp]
		\centering
		\includegraphics*[width=0.19\linewidth]{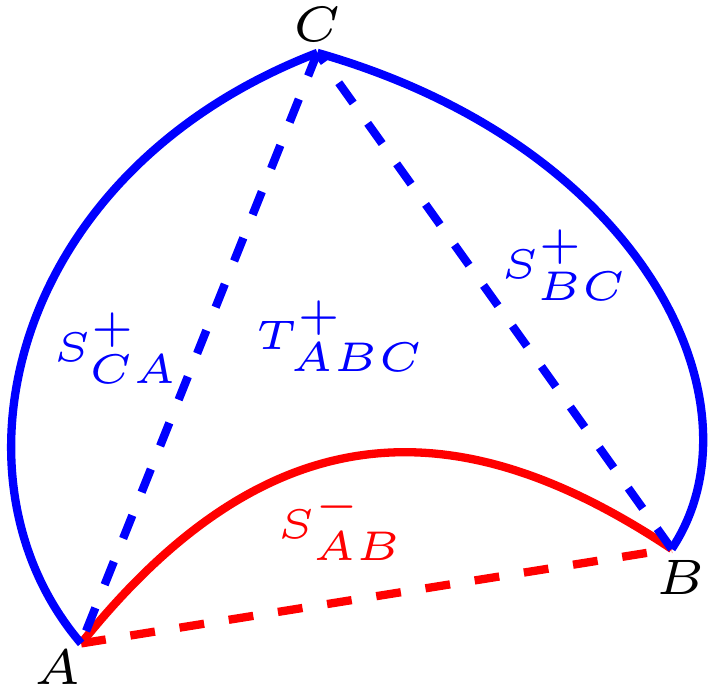}
		\hspace{0.5in}
		\includegraphics*[width=0.24\linewidth]{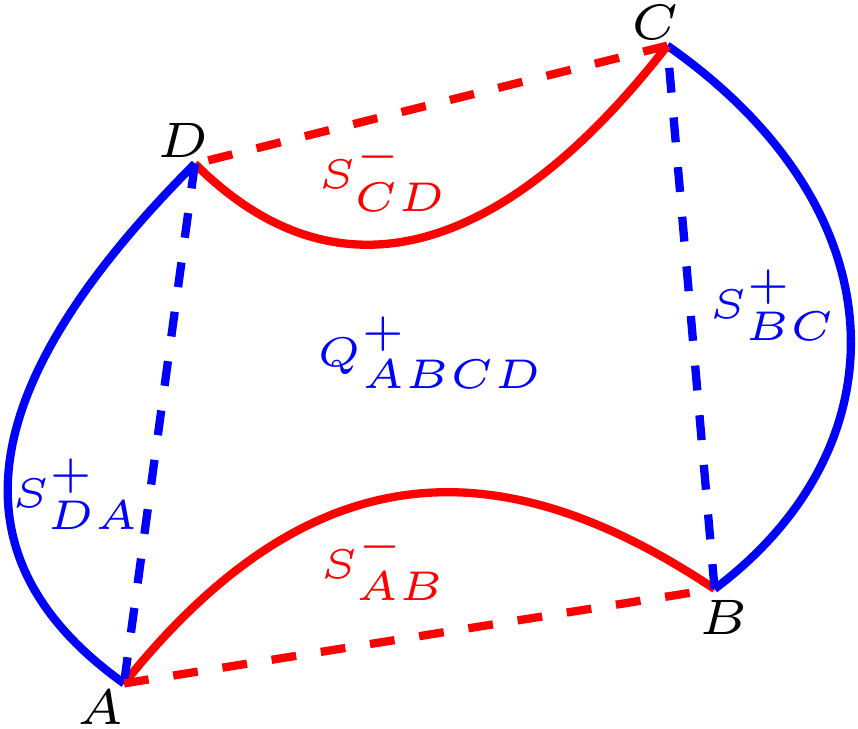}
		\caption{Examples for the partition of quadratic curvilinear triangle (left panel) and quadrilateral (right panel), the blue parts represent them in set $"+"$ and the red parts represent them in set $"-"$ in both panels.}
		\label{partition_2}
	\end{figure}

%	Take a curved quadrilateral in Fig.\ref{partition_2} (right) for an example, the group of $"+"$ comprises three separate parts: two parabolic segments and one quadrilateral. The $"-"$ group is made up by two parabolic segments. Generally speaking, there are consistently five components, with one component, which is a quadrilateral, in the group of $"+"$. In the case shown in the Fig.\ref{partition_2}, we can decompose the curved quadrilateral into
%	$
%	ABCD=\big(Q^+_{ABCD}\cup S^+_{BC}\cup S^+_{DA}\big)\backslash\big(S^-_{AB}\cup S^-_{CD}\big).
%	$
%	Then we can consider the integral on intersection between such curved quadrilaterals or curved triangles, the logic is inspired by the algorithm presented in \cite{warshaw1977area} and summarize in \cite{shashkov2024remapping}.
	Take a curved quadrilateral in Fig.\ref{partition_2} (right) for an example, we can decompose the curved quadrilateral into
	$
	ABCD=\big(Q^+_{ABCD}\cup S^+_{BC}\cup S^+_{DA}\big)\backslash\big(S^-_{AB}\cup S^-_{CD}\big).
	$
	Then we can consider the integral on intersection between such curved quadrilaterals or curved triangles, the logic is inspired by the algorithm presented in \cite{warshaw1977area} and summarize in \cite{shashkov2024remapping}, and it can be formulate as
	\begin{equation}\label{inters}
		\|A\cap B\|=\sum\limits_{k_A,k_B}\|A_{k_A}^+\cap B_{k_B}^+\|+\sum\limits_{l_A,l_B}\|A_{l_A}^-\cap B_{l_B}^-\|- \sum\limits_{k_A,l_B}\|A_{k_A}^+\cap B_{l_B}^-\|-\sum\limits_{l_A,k_B}\|A_{l_A}^-\cap B_{k_B}^+\|,
	\end{equation}
	where the units are represented as parts $\left\{  
	\begin{array}{l}
			A=\left(\cup_{k_A}A_{k_A}^+\right)\setminus\left(\cup_{l_A}A_{l_A}^-\right),\\  
			B=\left(\cup_{k_B}B_{k_B}^+\right)\setminus\left(\cup_{l_B}B_{l_B}^-\right).
		\end{array}  
	\right.$\\
	For applying this algorithm to DG framework on unstructured meshes, the traditional basis defined as in \cite{zhu2013runge} can no longer used, a new set of basis can be derived by a 2D Gram-Schmidt process over curvilinear unstructured meshes.

\subsection{Algorithm for Intersection of Convex Domains}
	We decompose complex, concave regions into simple convex regions, which makes it easier for us to utilize the two good properties of convex regions $\Omega_1$ and $\Omega_2$:
	\begin{enumerate}
		\item The intersection domain $\Omega_{12}=\Omega_1\cup\Omega_2$ is convex.
		\item The vertices of $\Omega_{12}$ include $\left\{\begin{array}{l}
			\text{intersection points of the domain boundaries,}\\
			\text{vertices contained within another domain.}
		\end{array}\right.$
	\end{enumerate}
	
\subsubsection{Integral on Intersecion of Convex Domains}\label{order}
	Since $\Omega_{12}$ is convex, we can order all vertices of $\Omega_{12}$ counter-clockwise by initially picking node $O$ that lies interior to the domain $\Omega_{12}$ (all vertices' convex blend). Subsequently, the angular coordinates of the vertex set are determined within reference frame defined by arbitrary axis $r$ and origin $O$. The vertices are then ordered based on their calculated polar angles. 
	Next, we only need to find out all the vertices of intersection domain $\Omega_{12}$.

	\subsubsection{The Vertexes of Intersection Domain}\label{vertexes}
	Two geometric operations are necessary to find all vertices of $\Omega_{12}$:
	\begin{enumerate}
		\item \textbf{Determine the intersections between the segments of one boundary and those of another boundary.}\\
		The upstream element is represented by a quadratic-curvilinear triangular approximation and the Eulerian mesh is composed of triangular elements, thus the intersection points of two straight lines, intersections of an arc and a straight line as well as intersections between arcs are within the scope of our consideration.\\
		For two straight lines, coordinate of the intersection point $(P_x,P_y)$ can be expressed easily. As for an arc with parametric equation \eqref{parametric} and a straight segment, we can find the intersection point by solving the system:
		$$\left\{ 
		\begin{array}{l}
			a^x\xi^2+b^x\xi+c^x=x_1+\eta(x_2-x_1),\\  
			a^y\xi^2+b^y\xi+c^y=y_1+\eta(y_2-y_1).
		\end{array}  
		\right.$$
%		After eliminating parameter $\eta$, we get a quadratic equation of parameter $\xi$, choosing the quadratic equation's roots which satisfy $\xi\in\left[0,1\right]$, and substitute it into the system of equations to find the solution of $\eta$, if $\eta\in\left[0,1\right]$, the corresponding point is considered as an intersection point.\\
        And the algorithm for finding intersections between two arcs is shown in Appendix \ref{appendixD}.
%		As for two arcs with parametric equations:
%		\begin{equation}\label{parametric_arcs}
%			\left\{ 
%			\begin{array}{l}
%				x_1(\xi)=a_1^x\xi^2+b_1^x\xi+c_1^x,\\  
%				y_1(\xi)=a_1^y\xi^2+b_1^y\xi+c_1^y,
%			\end{array}  
%			\right.\text{ and }\left\{  
%			\begin{array}{l}
%				x_2(\eta)=a_2^x\eta^2+b_2^x\eta+c_2^x,\\  
%				y_2(\eta)=a_2^y\eta^2+b_2^y\eta+c_2^y,
%			\end{array}  
%			\right.
%		\end{equation}
%		where $\xi,\eta\in\left[-1,1\right]$. Intersection points satisfy $\left\{  
%		\begin{array}{l}
%			x_1(\xi)=x_2(\eta)\\  
%			y_1(\xi)=y_2(\eta)
%		\end{array}  
%		\right.$, this system consists of two quadratic equations, by applying this, we can simplify the system to a single quartic equation:
%		$$\begin{array}{rl}
%			a_1^x\xi^2+b_1^x\xi&=a_2^x\eta^2+b_2^x\eta+c_2^x-c_1^x:=\alpha,\\
%			a_1^y\xi^2+b_1^y\xi&=a_2^y\eta^2+b_2^y\eta+c_2^y-c_1^y:=\beta.
%		\end{array}$$
%		We treat this system as  $\left[\begin{matrix}
%			\xi^2 \\
%			\xi \\
%		\end{matrix} \right]=\left[\begin{matrix}
%			a_1^x & b_1^x\\
%			a_1^y & b_1^y\\
%		\end{matrix} \right]^{-1}\left[\begin{matrix}
%			\alpha \\
%			\beta \\
%		\end{matrix} \right].$
%		By organizing coefficients, we have:
%		$$\begin{array}{rl}
%			\xi^2&=A\eta^2+B\eta+C,\\
%			\xi&=a\eta^2+b\eta+c.
%		\end{array}$$
%		This system yields a quartic equation in $\eta$ that
%		$$A\eta^2+B\eta+C=(a\eta^2+b\eta+c)^2.$$
%		So we find the intersection points of arcs by solving the equation.
		\item \textbf{Determine if a vertex of one domain is inside another domain.}\\
		The upstream element is approximated by a quadratic-curvilinear triangle, which can be divided into a triangle and three parabolic segments, showing in the Fig.\ref{partition_2} (left) thus we need to determine whether a vertex is inside of a triangle or a parabolic segment.\\
		As for the case of triangle, we can ascertain it along the sides of the triangle with induced orientation side-by-side by calculating the outer product.\\
%		As for the case of triangle, we can ascertain it along the sides of the triangle with induced orientation side-by-side by calculating the outer product $\boldsymbol{a}\times\boldsymbol{b}$, where $\boldsymbol{a}=(x_2-x_1,y_2-y_1),\boldsymbol{b}=(x-x_1,y-y_1)$, and $(x_1,y_1), (x_2,y_2)$ are the starting and ending points along the induced orientation respectively, $(x,y)$ is the vertex to be determined. By the knowledge of analytic geometry, $$\boldsymbol{a}\times\boldsymbol{b}=\left | \begin{matrix}
%			\boldsymbol{i}&\boldsymbol{j}&\boldsymbol{k}\\
%			x_2-x_1&y_2-y_1&0  \\
%			x-x_1&y-y_1&0 \\
%		\end{matrix} \right | = \left | \begin{matrix}
%			x_2-x_1&y_2-y_1  \\
%			x-x_1&y-y_1 \\
%		\end{matrix} \right | \boldsymbol{k},$$
%		If $\boldsymbol{a}\times\boldsymbol{b}$ and $\boldsymbol{k}$ are in the same direction (i.e. $(y-y_1)(x_2-x_1)-(y_2-y_1)(x-x_1)>0$), vertex $(x,y)$ is on the left side of the segment. And if it established for all edges, the vertex is inside the triangle.\\
		As for the case of parabolic segment, the algorithm is shown in Appendix \ref{appendixE}.%a positive area is obtained from the contour integral under the assumption that the boundary of it has a counter-clockwise orientation.\\
	\end{enumerate}
	\subsection{Procedure of The Intersection-Based Remapping Algorithm}
To remap the solution from the Eulerian mesh to the upstream element, we proceed as follows.
	\begin{enumerate}
		\item \textbf{Search step.} To identify Eulerian elements (black lines in Fig.~\ref{auxiliary_mesh}, left) that intersect the Lagrangian element $K_j^\star$, we build an auxiliary rectangular mesh with a lookup index~\cite{cai2022eulerian} (right). The index returns candidate elements that may intersect $K_j^\star$; computing intersections between edges of $K_j^\star$ and these candidates yields the final set of intersecting Eulerian elements.
		\begin{figure}[ht]
			\centering
			\includegraphics*[width=0.59\linewidth]{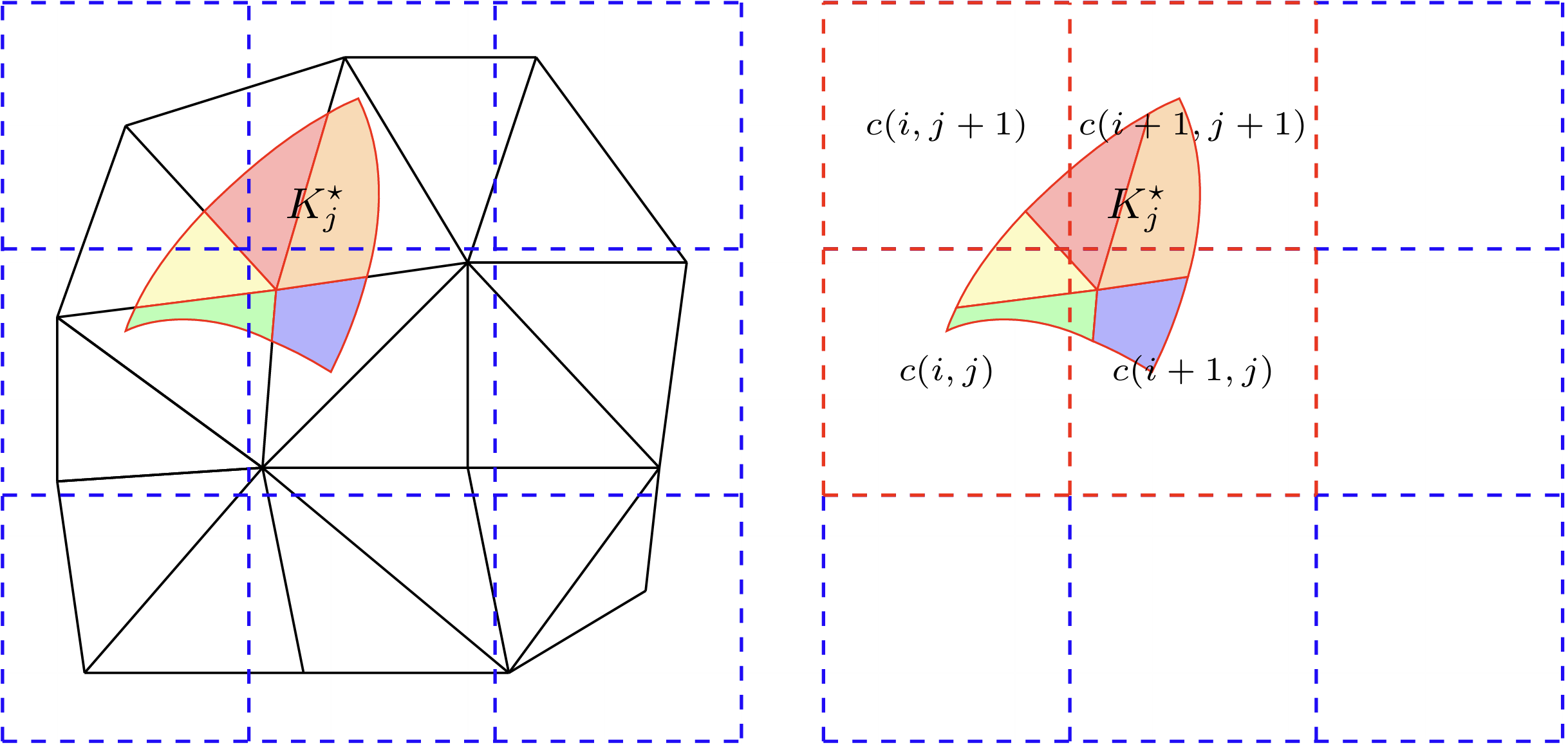}
\includegraphics*[width=0.30\linewidth]{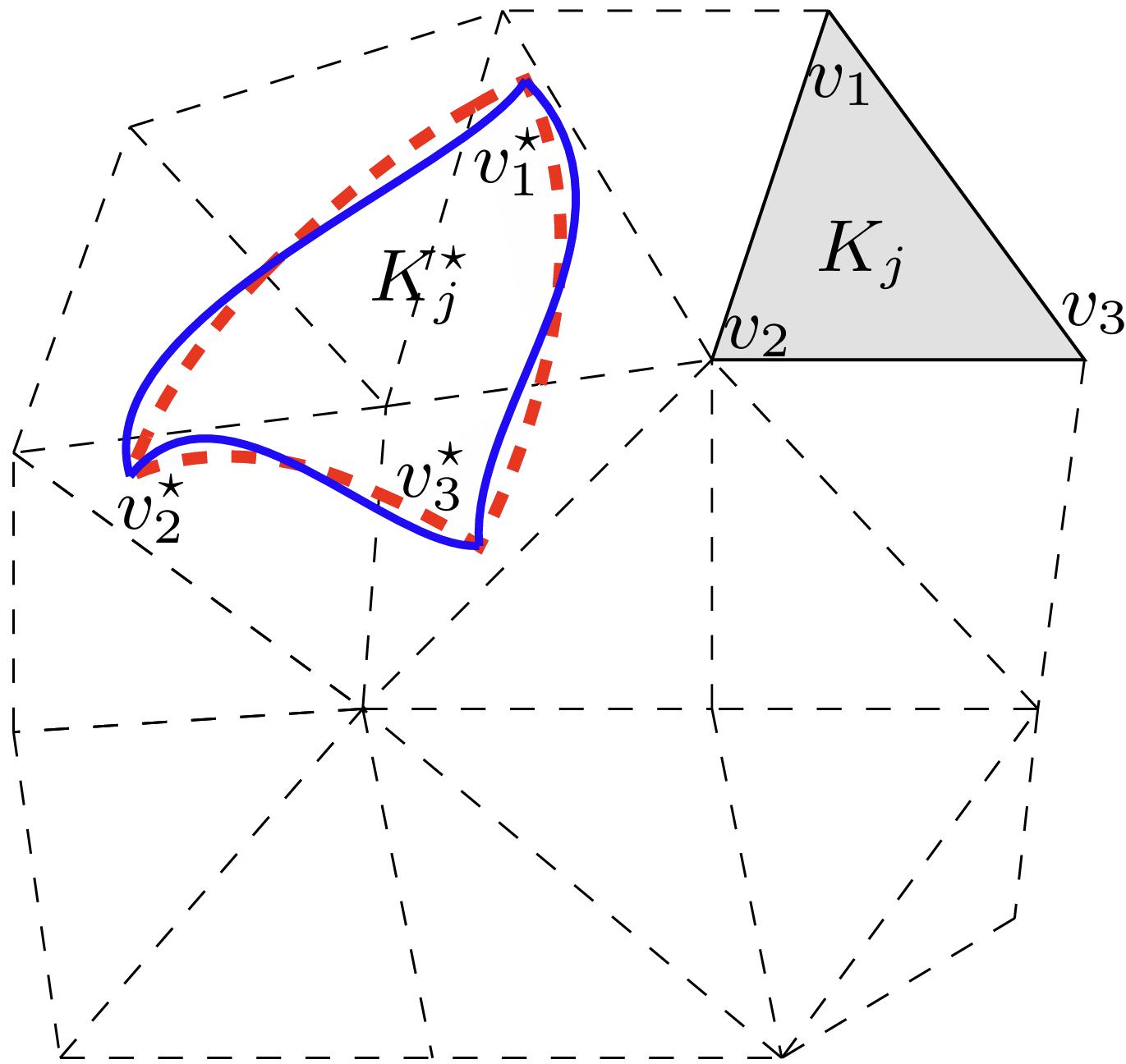}
			\caption{Left: auxiliary mesh (blue lines) connectS with Eulerian mesh (black lines) and the upstream element $K_j^\star$. Middle: the look-up index that locate $K_j^\star$ on auxiliary mesh with red lines. Right: an element $K_j$ on Eulerian mesh and its upstream element $K_j^\star$ which exactly bounded by blue curves and approximated by red dashed curves.}
			\label{auxiliary_mesh}
		\end{figure}

		\item \textbf{Partition step.} Segment concave regions into convex regions. For quadratic-curvilinear triangle or quadrilateral, we can segment them as shown in Fig.\ref{partition_2}. For cubic or higher-order curvilinear elements, more segmentation are needed.
		\item \textbf{Clipping step.} Perform a clipping algorithm for $K_j^\star$ and elements on Eulerian mesh to get the intersection regions shown in Fig.\ref{auxiliary_mesh} with colorful parts following the algorithm summarized in Appendix \ref{appendixF}.

		\item \textbf{Integration step.} By introducing auxiliary functions $P(x,y)$ and $Q(x,y)$, the area integral can be converted into line integrals via Green's theorem, same as the procedure in Section \ref{procedure_sldg}.
	\end{enumerate}
	\remark{For cubic or higher-order curvilinear elements, we need to use the inflection point of the curve which can be found be analyse the second derivative of the parameter equation:
		\begin{equation}\label{derivative}
			\frac{\mathrm{d}^2 y }{\mathrm{d} x^2}=\frac{\frac{\mathrm{d}}{\mathrm{d}\xi}\left(\frac{\mathrm{d}y}{\mathrm{d}x}\right)}{\frac{\mathrm{d}x}{\mathrm{d}\xi}}=\frac{\frac{\mathrm{d}}{\mathrm{d}\xi}\left(\frac{\mathrm{d}y}{\mathrm{d}\xi}/\frac{\mathrm{d}x}{\mathrm{d}\xi}\right)}{\frac{\mathrm{d}x}{\mathrm{d}\xi}}=\frac{\frac{\mathrm{d}^2y}{\mathrm{d}\xi^2}\frac{\mathrm{d}x}{\mathrm{d}\xi}-\frac{\mathrm{d}^2x}{\mathrm{d}\xi^2}\frac{\mathrm{d}y}{\mathrm{d}\xi}}{\left(\frac{\mathrm{d}x}{\mathrm{d}\xi}\right)^3}.
		\end{equation}
		Take the cubic-curvilinear trianguler element as an example, the parametric equations are $\left\{ 
		\begin{array}{l}
			x(\xi)=a^x\xi^3+b^x\xi^2+c^x\xi+d^x,\\ 
			y(\xi)=a^y\xi^3+b^y\xi^2+c^y\xi+d^x.
		\end{array}  
		\right.$ By applying the parametric equations to equation \eqref{derivative}, we have the second derivative
		$$
		\frac{\mathrm{d}^2 y }{\mathrm{d} x^2}=\frac{6\left(a^yb^x-a^xb^y\right)\xi^2+6\left(a^yc^x-a^xc^y\right)\xi+2\left(b^yc^x-b^xc^y\right)}{\left(3a^x\xi^2+2b^x\xi+c^x\right)^3}.
		$$
		By applying the Fundamental Theorem of Algebra, we know that there are no more than two inflection point on each side of cubic triangle, and the inflection points partition the cubic triangle into convex domains as shown in Fig.\ref{cubic_triangle}. From the analysis in Section~\ref{section:analytical}, we know that the dividing line and boundary of element will not intersect with appropriately small element size.
		\begin{figure}[ht]
			\centering
			\includegraphics*[width=0.18\linewidth]{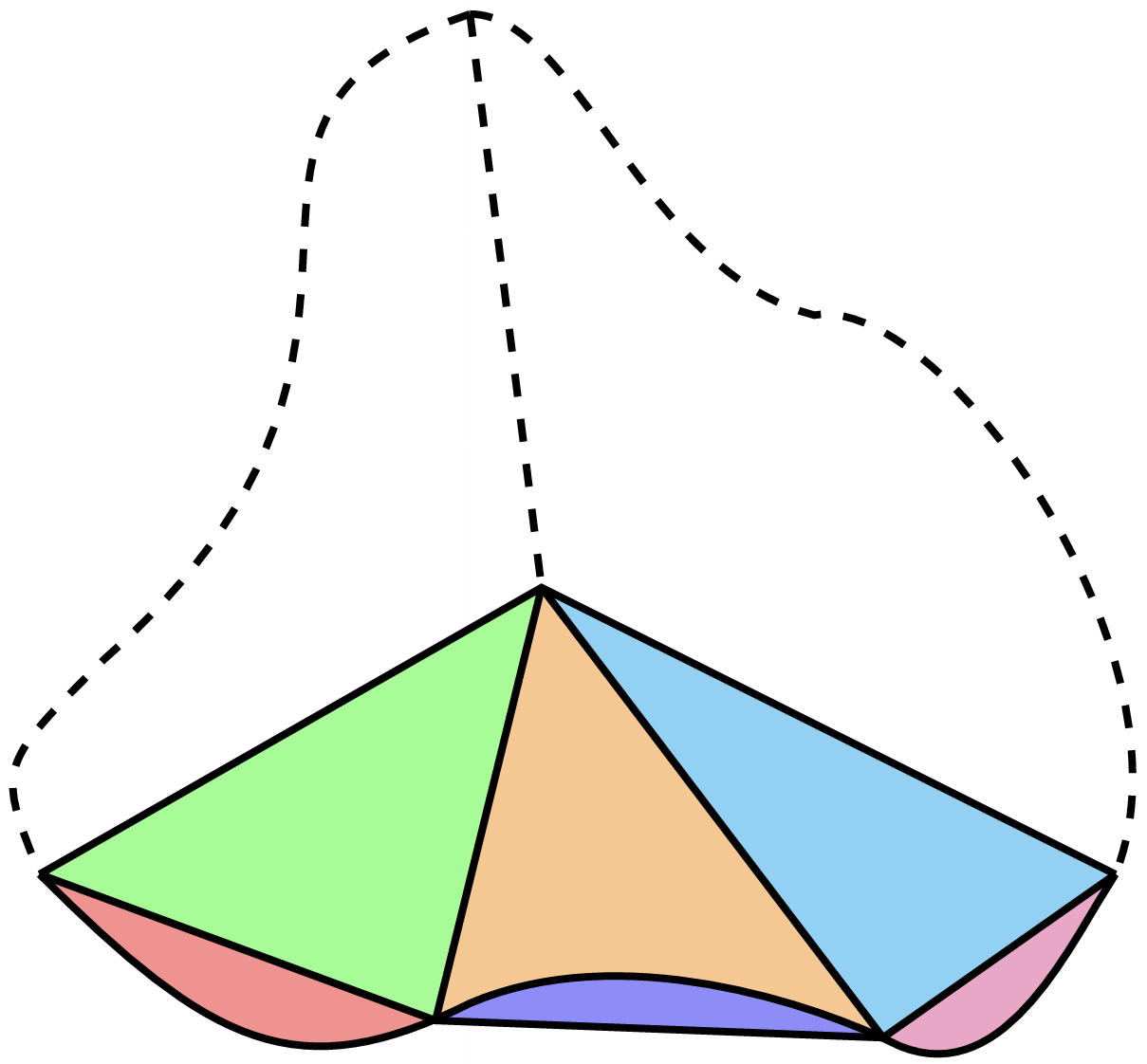}
			\caption{Take third of a cubic triangle (colorful region) as example. The region be divided into six convex domains.}
			\label{cubic_triangle}
	\end{figure}}\label{remark3}

	\section{Consistent Error Analysis} \label{section:analytical}
This section study the consistency of the proposed scheme on unstructured meshes and other properties, like mass conservation. 
%To verify the convergence of the scheme, a rigorous consistency proof is provided in this section, and stability is numerically verified with large final time and arbitrarily large time-stepping sizes simulations in section \ref{section:numerical}.
\begin{proposition}\label{mass_conserve}
	(Mass Conservation). The proposed scheme is mass conservative if periodic boundary condition is imposed, which can be expressed as
	\begin{equation}\label{mass_conservation}
		\int_{\Omega}u^{n+1}(x,y)\mathrm{d}x\mathrm{d}y=\int_{\Omega}u^{n}(x,y)\mathrm{d}x\mathrm{d}y.
	\end{equation}
\end{proposition}

\begin{proof}
	The proof follows the standard SLDG argument (see \cite{cai2022eulerian}) and is omitted. 
\end{proof}

We consider an edge of $K_j$ on Fig.\ref{auxiliary_mesh} (right), $\overline{v_1v_2}$, at $t=t^{n+1}$ and its characteristic upstream curve $\widehat{v^\star_1v^\star_2}$. Here, the blue curves is obtained by exactly solving equation \eqref{characteristic} from every points on $\overline{v_1v_2}$ as final value. This is impractical to implement, so we approximate them by the red quadratic curves.\\
Suppose the coordinate of $v_i$ is $(x_i,y_i)$, $i=1,2,3$, then we have the parametric equation of $\overline{v_1v_2}$: $\left\{ 
\begin{array}{l}
x(\theta)=x_1+\theta\Delta x:=x_1+\theta(x_2-x_1),\\  
y(\theta)=y_1+\theta\Delta y:=y_1+\theta(y_2-y_1),
\end{array}  
\right.$ where $\theta\in [0,1]$. We define points on $\widehat{v^\star_1v^\star_2}$ as $\mathcal{X}(\theta):=X\Big(x(\theta),y(\theta);t^n\Big)$ and $\mathcal{Y}(\theta):=Y\Big(x(\theta),y(\theta);t^n\Big)$, where $X$ and $Y$ represent the solution of equation \eqref{characteristic} from $t^{n+1}$ to $t^n$ and $\Big(x(\theta),y(\theta)\Big)$ as final value. From that process, the parametric equation of $\widehat{v^\star_1v^\star_2}$ is represented by
\begin{equation}\label{parametric_equation}
\left\{ 
\begin{array}{l}
\mathcal{X}(\theta)=x(\theta)+\int_{t^{n+1}}^{t^n}a\Bigg(X\Big(x(\theta),y(\theta);t\Big),Y\Big(x(\theta),y(\theta);t\Big),t\Bigg)dt,\\  
\mathcal{Y}(\theta)=y(\theta)+\int_{t^{n+1}}^{t^n}b\Bigg(X\Big(x(\theta),y(\theta);t\Big),Y\Big(x(\theta),y(\theta);t\Big),t\Bigg)dt,
\end{array}  
\right.
\end{equation}
where $\theta\in[0,1]$.\\ 
For analyzing the consistency of the SLDG method on unstructured meshes, we need some lemmas, and we divide our proof in three steps. First, we need to verify some property of the upstream element, including Lemma \ref{lemma_derivative}, \ref{lemma_derivative2} and \ref{lemma_distance}, as a prerequisite for following proof. The next thing to do in the proof is to estimate the error generated by piecewise polynomials approximation in Lemmas \ref{projection_theorem} and \ref{projection_lemma}. Another step is to estimate the approximation error of the upstream element in Lemma \ref{edge_approximate}. Finally, we can combine the Lemmas and prove the consistency of the proposed SLDG scheme in Proposition \ref{consistency}.
\begin{lemma}\label{lemma_derivative}
If $a(x,y,t),b(x,y,t)\in C^{\infty}$, the derivatives of $\mathcal{X}(\theta)$ and $\mathcal{Y}(\theta)$ satisfies
\begin{equation}\label{derivative1}
\left[\begin{matrix}
	\frac{\mathrm{d} \mathcal{X}(\theta)}{\mathrm{d} \theta} \\
	\frac{\mathrm{d} \mathcal{Y}(\theta)}{\mathrm{d} \theta} \\
\end{matrix} \right]=\left[\begin{matrix}
	\Delta x+o(\Delta t)\\
	\Delta y+o(\Delta t)\\
\end{matrix} \right],\quad \theta\in[0,1].
\end{equation}
\end{lemma}

\begin{lemma}\label{lemma_derivative2}
If $a(x,y,t),b(x,y,t)\in C^{\infty}$ and $\Delta t\sim\Delta x\sim\Delta y$, the high-order derivatives of $\mathcal{X}(\theta)$ and $\mathcal{Y}(\theta)$ satisfies
$$\left[\begin{matrix}
\frac{\mathrm{d}^k \mathcal{X}(\theta)}{\mathrm{d} \theta^k} \\
\frac{\mathrm{d}^k \mathcal{Y}(\theta)}{\mathrm{d} \theta^k} \\
\end{matrix} \right]=\left[\begin{matrix}
O(\Delta t^k)\\
O(\Delta t^k)\\
\end{matrix} \right],\quad \theta\in[0,1],\ \ k=2,3...$$
\end{lemma}

The Lemma \ref{lemma_derivative} and \ref{lemma_derivative2} guarantee that the exact upstream elements are only mildly distorted, so the sides of upstream elements are Jordan curves.

\begin{lemma}\label{lemma_distance}
We assume the coordinate of $v^\star_i$ is $(x^\star_i,y^\star_i),\ \ i=1,2,3$. If $a(x,y,t)$, $b(x,y,t)\in C^{\infty}$ and $h\sim\Delta t\sim\Delta x\sim\Delta y$, we have
\begin{equation}\label{upstream}
x^\star_2-x^\star_1=O(h),\ \ y^\star_2-y^\star_1=O(h),\ \ d(v^\star_1,v^\star_2)=O(h),
\end{equation}
where $d(\cdot,\cdot)$ denotes Euclidean distance.
\end{lemma}
%\begin{proof}
%It can easily be proved by applying Lemma \ref{lemma_derivative}. Integrate both sides of \eqref{derivative1} with respect to $\theta$ over $[0,1]$, we have
%\begin{equation}\label{upstream1}
%x^\star_2-x^\star_1=\int_{0}^{1}\frac{\mathrm{d}\mathcal{X}(\theta)}{\mathrm{d}\theta}\mathrm{d}\theta=\int_{0}^{1}\Delta x+o(\Delta t)\mathrm{d}\theta=\Delta x+o(\Delta t).
%\end{equation}
%Since $h\sim\Delta t\sim\Delta x$, we can express $\Delta x+o(\Delta t)$ as $O(h)$. The proof of $y^\star_2-y^\star_1$ is the same as \eqref{upstream1}. The distance between $v^\star_1$ and $v^\star_2$ is
%$$d(v^\star_1,v^\star_2)=\sqrt{(x^\star_2-x^\star_1)^2+(y^\star_2-y^\star_1)^2}=\sqrt{O(h^2)}=O(h).$$
%The proof of Lemma \ref{lemma_distance} is completed.
%\end{proof}

We recall the following classical result without proof; see, e.g., \cite{yosida2012functional}.
\begin{theorem}\label{projection_theorem}
(Projection Theorem). Let $\mathbb{H}$ be a Hilbert space, and let $\mathbb{M}$ be a closed subspace of $\mathbb{H}$. For any given vector $b$, there exists a unique $m^\star\in\mathbb{M}$, such that for any $m\in\mathbb{M}$
\begin{equation}\label{projection_theorem1}
\left\lVert m^\star-b\right\rVert\leqslant\left\lVert m-b\right\rVert,
\end{equation}
where the norm is induced by inner product. Furthermore, the necessary and sufficient condition of the uniqueness of $m^\star$ and equation \eqref{projection_theorem1} is $$(m^\star-b)\perp\mathbb{M}.$$
\end{theorem}

\begin{lemma}\label{projection_lemma}
Denote $Pu\in V_h^k$ as the projection of $u$ on $V_h^k$, which is for any $v\in V_h^k$, $(Pu-u,v)_{\mathbb{H}}=0$, where $(\cdot,\, \cdot)_{\mathbb{H}}$ denotes the inner product of $\mathbb{H}$. Specifically, we define the inner product as
\begin{equation}\label{projection_vector}
(Pu-u,v)_{\mathbb{H}}=\int_{K_j}\left(Pu-u\right)vdx=0.
\end{equation}
Then there exists a constant $C$, such that
\begin{equation}\label{accuracy_order}
\left\lVert Pu-u\right\rVert\leqslant Ch^{k+1}.
\end{equation}
\end{lemma}

To numerically approximate $\widehat{v^\star_1v^\star_2}$ shown in Fig. \ref{auxiliary_mesh} (right) as the blue curve, we defined parametric equations \eqref{parametric}, the red curve $\overline{v^\star_1v^\star_2}$ in Fig. \ref{auxiliary_mesh} (right). In the $k$-order polynomial space, the parametric equations can be express as
\begin{equation}\label{parametric2}
\mathbf{x}^F(\theta)=\left[\begin{matrix}
x^F \\
y^F \\
\end{matrix} \right]=\left[\begin{matrix}
\alpha^{x,F}_{k} \\
\alpha^{y,F}_{k} \\
\end{matrix} \right]\theta^k+\cdot\cdot\cdot+\left[\begin{matrix}
\alpha^{x,F}_{1} \\
\alpha^{y,F}_{1} \\
\end{matrix} \right]\theta+\left[\begin{matrix}
\alpha^{x,F}_{0} \\
\alpha^{y,F}_{0} \\
\end{matrix} \right],\ \ \theta\in[0,1].
\end{equation}

\begin{lemma}\label{edge_approximate}
If $a(x,y,t)$, $b(x,y,t)\in C^{\infty}$ and $h\sim\Delta x\sim\Delta y$, there exist constant $C_1,C_2$, such that
\begin{equation}\label{xy_h^k+1}
\left\lVert\mathcal{X}(\theta)-x^F(\theta)\right\rVert\leqslant C_1h^{k+1},\ \ \left\lVert\mathcal{Y}(\theta)-y^F(\theta)\right\rVert\leqslant C_2h^{k+1}.
\end{equation}
Furthermore, the area between $\widehat{v^\star_1v^\star_2}$ and $\overline{v^\star_1v^\star_2}$, 
\begin{equation}\label{d_h^k+1}
d(\widehat{v^\star_1v^\star_2},\overline{v^\star_1v^\star_2}):=\int_{0}^{1}\sqrt{\left(\mathcal{X}-x^F\right)^2+\left(\mathcal{Y}-y^F\right)^2}\mathrm{d}\theta\leqslant Ch^{k+1},
\end{equation}
for a constant $C$.
\end{lemma}

\begin{theorem}\label{consistency}
(Consistent error analysis). The numerical solution  $u^{n+1}(x,y)\in V_h^2$ over $K_j$ satisfies
\begin{equation}
\left\lVert u^{n+1}(x,y)-u(x,y,t^{n+1}) \right\rVert_{L^2}=O(h^3)
\end{equation}
where the norm denotes the $L^2$ norm over $K_j$, with $u(x,y,t)$, $a(x,y,t)$, $b(x,y,t)\in C^{\infty}$ and $h\sim\Delta t\sim\Delta x\sim\Delta y$ if $u^n(x,y)$ is a third order approximation of analytical solution at $t^n$ time level.
\end{theorem}

	\section{Numerical Experiments}\label{section:numerical}
	This section presents some numerical experiments to demonstrate the effectiveness of  the SLDG scheme on curvilinear unstructured triangular meshes. 
%A comparison is made between the numerical results and the exact solution, and the convergence rates are computed. 
We assume the computational domain $\Omega$ is devided into $N$ elements. %, which are quasi-uniform and unstructured triangular elements. 
The refinement of the mesh is achieved by joining the midpoints of the edges of the elements, that is, the original element is divided into four sub-elements, and the new mesh point is the midpoint of the edge of the original element. For the sake of simplicity, we take $r_j=\frac{2|K_j|}{|\partial K_j|}$ as the radius pertaining to the inscribed circle within the element $K_j$ as in \cite{cai2022eulerian,luo2019quasi}, where $|K_j|$ and $|\partial K_j|$ are the area and the perimeter of the element $K_j$ respectively. Unless otherwise specified herein, the time step is chosen as
	$$\Delta t=\text{CFL}\cdot\frac{\min_j r_j}{\max_j\max_{face}|\mathbf{V\cdot n}|},$$
	where  $\mathbf{V}$ is the velocity field, and $\mathbf{n}$ is the outward unit normal vector of the face.

\begin{example}
 (Rigid body rotation on a circle domain). 
	Consider the rigid body rotation transport equation on circle domain as Fig.\ref{circle} (left).
	\begin{equation}\label{rigid_body_rotation}
		u_t-(yu)_x+(xu)_y=0,\ \ (x,y)\in \{(x,y)|x^2+y^2 \leqslant\pi^2\},
	\end{equation}
	with the initial condition $u(x,y,0)=e^{-3x^2-3y^2}$. Due to the symmetry of initial values and computational regions, the analytical solution remains unchanged in terms of computation time $T=2\pi$. 
\end{example}

	We apply the proposed $P^k$ SLDG method on a circular area with unstructured mesh to solve this problem. 
 Table \ref{u_t-(yu)_x+(xu)_y=0} reports mesh-refinement results, showing second- and third-order accuracy for the  $P^1$ and $P^2$ schemes, respectively.
 To numerically prove the stability, we perform a numerical solution for \eqref{rigid_body_rotation} with large final time $T=50\pi$ and the numerical results are shown in Table \ref{rigid_body_rotation_T=50}.
 
	\begin{table}[htbp]
		\centering
		\caption{$L^1$, $L^2$ and $L^{\infty}$ errors along with the associated the accuracy orders for the 2D SLDG for rigid body rotation with CFL = 10 and final time $T=2\pi$.}
		\label{u_t-(yu)_x+(xu)_y=0}
		\begin{tabular}{lcccccccc}
			\toprule
			\textbf{$P^k$} & \textbf{$M$} & \textbf{$r_{max}$} & \textbf{$L^1$-error} & \textbf{order} & \textbf{$L^2$-error} & \textbf{order}  & \textbf{$L^{\infty}$-error} & \textbf{order} \\ \midrule
%			& 160 & 0.781 & 2.63E-02 &  & 6.78E-02 & & 6.76E-02 &  \\
%			$P^0$ & 522 & 0.392 & 1.62E-02 & 0.82 & 4.58E-02 & 0.76 & 6.06E-02 & 0.19 \\
%			& 1884 & 0.198 & 9.17E-03 & 0.88 & 2.58E-02 & 0.89 & 4.07E-02 & 0.62 \\ \midrule
			& 160 & 1.521 & 3.82E-03 &  &     1.13E-02 &  &     8.88E-02 & \\
			$P^1$ & 522 & 0.750 & 1.18E-03 & 1.99 & 4.08E-03 & 1.72 & 6.67E-02 & 0.49 \\
			& 1884 & 0.396 & 3.16E-04 & 2.05 & 1.15E-03 & 1.97 & 2.45E-02 & 1.56 \\
			& 7432 & 0.189 & 8.00E-05 & 2.00 & 2.90E-04 & 2.01 & 6.58E-03 & 1.92 \\ \midrule
			& 160 & 1.521 & 7.07E-04 &  & 2.44E-03 &  & 3.58E-02 & \\
			$P^2$ & 522 & 0.750 & 1.01E-04 & 3.29 & 3.67E-04 & 3.20 & 5.32E-03 & 3.22 \\
			& 1884 & 0.396 & 1.52E-05 & 2.95 & 5.68E-05 & 2.91 & 1.08E-03 & 2.49 \\
			& 7432 & 0.189 & 2.07E-06 & 2.90 & 7.78E-06 & 2.90 & 1.51E-04 & 2.87 \\
			\bottomrule
		\end{tabular}
	\end{table}

	\begin{table}[!htbp]
		\centering
		\caption{$L^1$, $L^2$ and $L^{\infty}$ errors and the associated the accuracy orders for the 2D SLDG for constant-coefficient transport equation with $\text{CFL} = 1$ and final time $T=50\pi$.}
		\label{rigid_body_rotation_T=50}
		\begin{tabular}{lcccccccc}
			\toprule
			\textbf{$P^k$} & \textbf{$M$} & \textbf{$r_{max}$} & \textbf{$L^1$-error} & \textbf{order} & \textbf{$L^2$-error} & \textbf{order}  & \textbf{$L^{\infty}$-error} & \textbf{order} \\ \midrule
			& 160 & 1.521 & 1.01E-02 &  & 2.81E-02 &  & 9.01E-02 & \\
			$P^1$ & 522 & 0.750 & 3.02E-03 & 2.04 & 1.02E-02 & 1.72 & 4.27E-02 & 1.26 \\
			& 1884 & 0.396 & 5.46E-04 & 2.67 & 2.02E-03 & 2.52 & 1.16E-02 & 2.03 \\
			& 7432 & 0.189 & 1.01E-04 & 2.46 & 3.85E-04 & 2.42 & 4.08E-03 & 1.52 \\ \midrule
			& 160 & 1.521 & 1.94E-03 &  & 6.89E-03 &  & 3.89E-02 & \\ 
			$P^2$ & 522 & 0.750 & 1.41E-04 & 4.43 & 5.18E-04 & 4.38 & 5.85E-03 & 3.21 \\
			& 1884 & 0.396 & 1.97E-05 & 3.07 & 7.50E-05 & 3.01 & 1.14E-03 & 2.54 \\
			& 7432 & 0.189 & 2.23E-06 & 3.17 & 8.27E-06 & 3.21 & 1.54E-04 & 2.92 \\
			\bottomrule
		\end{tabular}
	\end{table}

\begin{example}
(Swirling deformation flow). 
	We consider 
	\begin{equation}\label{swirling_deformation_flow}
		u_t-\big(\cos^2(\frac{x}{2})\sin(y)g(t)u\big)_x+\big(\sin(x)\cos^2(\frac{y}{2})g(t)u\big)_y=0,
	\end{equation}
	where $g(t)=\cos(\frac{\pi t}{T})\pi$ and $(x,y)\in \{(x,y)|x^2+y^2\leqslant\pi^2\}$. We establish the initial condition to be the subsequent smooth cosine bell:
	$$
	u(x,y,0)=\left\{
	\begin{array}{ll}
		r_0^b\cos^6(\frac{r^b\pi}{2r_0^b}),\ \ &if r^b<r^b_0 ,\\
		0,\ \ &otherwise,
	\end{array}
	\right.
	$$
	where $r^b=\sqrt{(x-x_0^b)^2+(y-y_0^b)^2}$ denotes the distance between $(x,y)$ in the computational domain and the center of the cosine bell $(x_0^b,y_0^b)=(0.45\pi,0)$ and $r_0^b=0.45\pi$. \end{example}

	\begin{figure}[htbp]
		\centering
		\includegraphics[scale=0.27]{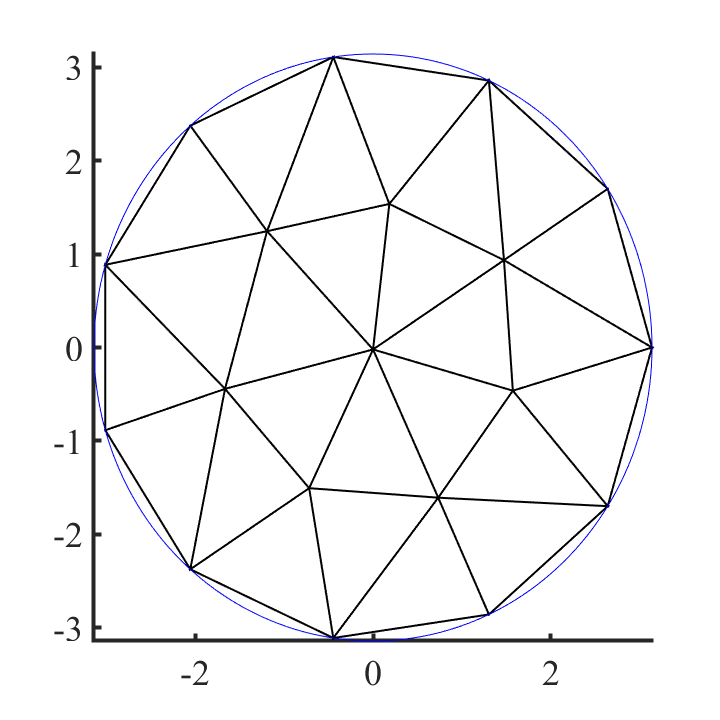}
		\hspace{0in}
		\includegraphics[scale=0.27]{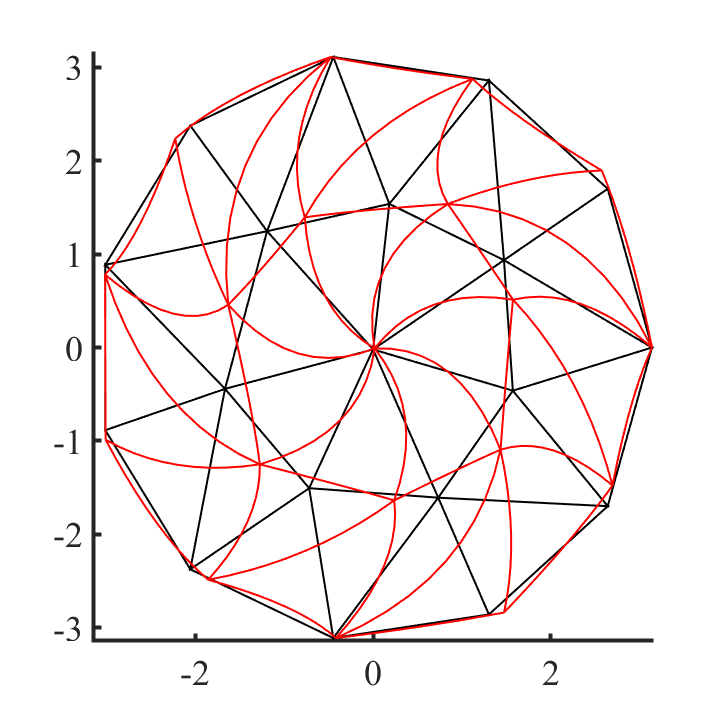}
		\caption{Initial triangular meshes on a circle domain with 25 elements (left) and corresponding upstream meshes with $\Delta t=\frac{\pi}{2}$ of red curves (right).}
		\label{circle}
	\end{figure}
	In the simulations, we set $CFL=10.5$ and the solution is computed up to $T=2\pi$. In Fig.\ref{circle}, left panel shows the background mesh on a circle domain with 25 elements and right panel shows the upstream elements with $\Delta t=\frac{\pi}{2}$. From the right panel of Fig.\ref{circle}, we can see some fine and complex structures that need to be captured. To avoid incorrect topology, we must account for the partial overlapping of edges between the background and upstream meshes and the close proximity of their edge endpoints.
	Table \ref{swirling deformation flow} summarizes a convergence study for $P^k\ (k=1,2)$ SLDG method in terms of the $L^1$, $L^2$ and $L^{\infty}$ errors and the associated orders of accuracy. Under mesh refinement, the $P^1$ SLDG scheme has second-order accuracy with $L^1$, $L^2$ and $L^{\infty}$ errors which is similar with the existing scheme with straight-sided upstream triangle, and more importantly, the $P^2$ SLDG scheme has third-order accuracy with $L^1$, $L^2$ and $L^{\infty}$ errors, which obtains more accurate solution functions than scheme in \cite{cai2017high,guo2014conservative}.\\
	We emphasize the difference between the last two parts of the Table \ref{swirling deformation flow} which are marked as red. We denote the two-dimensional SLDG scheme with upstream element of triangle on $P^2$ as $P^2_s$. By observation, we find that the $L^1$, $L^2$ and $L^{\infty}$ errors on all grids of the scheme with upstream element of quadratic curvilinear triangle are all smaller than that of the scheme with upstream element of triangle, and by contrast, the orders of accuracy of the scheme with upstream element of quadratic curvilinear triangle is more stable and more in line with third-order accuracy.

	\begin{table}[htbp]
		\centering
		\caption{$L^1$, $L^2$ and $L^{\infty}$ errors along with the associated the accuracy orders for the 2D SLDG for swirling deformation flow with $CFL = 10.5$ and $T=1.5$.}
		\label{swirling deformation flow}
		\begin{tabular}{lcccccccc}
			\toprule
			\textbf{$P^k$} & \textbf{$M$} & \textbf{$r_{max}$} & \textbf{$L^1$-error} & \textbf{order} & \textbf{$L^2$-error} & \textbf{order} & \textbf{$L^{\infty}$-error} & \textbf{order} \\ \midrule
			\multicolumn{9}{c}{2D SLDG scheme with quadratic curvilinear triangle}            \\ \midrule
%			& 160 & 0.781 & 1.42E-02 &  & 6.65E-02 &  & 1.21E-01 &  \\
%			$P^0$ & 522 & 0.392 & 1.22E-02 & 0.36 & 5.43E-02 & 0.34 & 8.68E-02 & 0.56 \\
%			& 1884 & 0.198 & 6.59E-03 & 0.96 & 3.48E-02 & 0.69 & 6.10E-02 & 0.55 \\ \midrule
            & 160 & 1.521 & 3.55E-03 &  & 1.48E-02 &  & 2.14E-01 & \\
			$P^1$ & 522 & 0.750 & 1.22E-03 & 1.80 & 5.19E-03 & 1.77 & 4.47E-02 & 2.65 \\
			& 1884 & 0.396 & 1.59E-04 & 3.18 & 6.89E-04 & 3.15 & 9.83E-03 & 2.36 \\
			& 7432 & 0.189 & 3.55E-05 & 2.18 & 1.66E-04 & 2.07 & 3.59E-03 & 1.47 \\ \midrule
			& 160 & 1.521 & 1.13E-03 & & 4.11E-03 &  & 3.78E-02 & \\
			{\color{red}$P^2$} & 522 & 0.750 & 1.28E-04 & 3.67 & 5.58E-04 & 3.38 & 6.36E-03 & 3.01 \\ 
			& 1884 & 0.396 & 1.09E-05 & 3.84 & 4.86E-05 & 3.80 & 9.61E-04 & 2.94 \\
			& 7432 & 0.189 & 1.36E-06 & 3.03 & 6.33E-06 & 2.97 & 1.62E-04 & 2.60 \\ \midrule
			\multicolumn{9}{c}{2D SLDG scheme with triangle}            \\ \midrule
			& 160 & 1.521 & 2.41E-03 &  & 1.01E-02 &  & 8.86E-02 & \\
			{\color{red} $P^2_s$} & 522 & 0.750 & 7.27E-04 & 2.03 & 4.12E-03 & 1.53 & 8.67E-02 & 0.04 \\ 
			& 1884 & 0.396 & 1.32E-04 & 2.66 & 7.80E-04 & 2.59 & 1.21E-02 & 3.07 \\
			& 7432 & 0.189 & 3.52E-05 & 1.93 & 2.14E-04 & 1.88 & 3.34E-03 & 1.88 \\
			\bottomrule
		\end{tabular}
	\end{table}

	After that, since the consistency of the proposed SLDG schemes on unstructured meshes has been analyzed, we study the numerical stabilities of it. We present the plots of $L^1$ and $L^2$ errors versus CFL of the SLDG method in Fig.\ref{fig_swirling_cfl} as in \cite{cai2022eulerian}. From the figure, we can see that the proposed SLDG schemes on unstructured meshes is stable for arbitrarily large time-stepping sizes.
	\begin{figure}[htbp]
		\centering
		\includegraphics[scale=0.27]{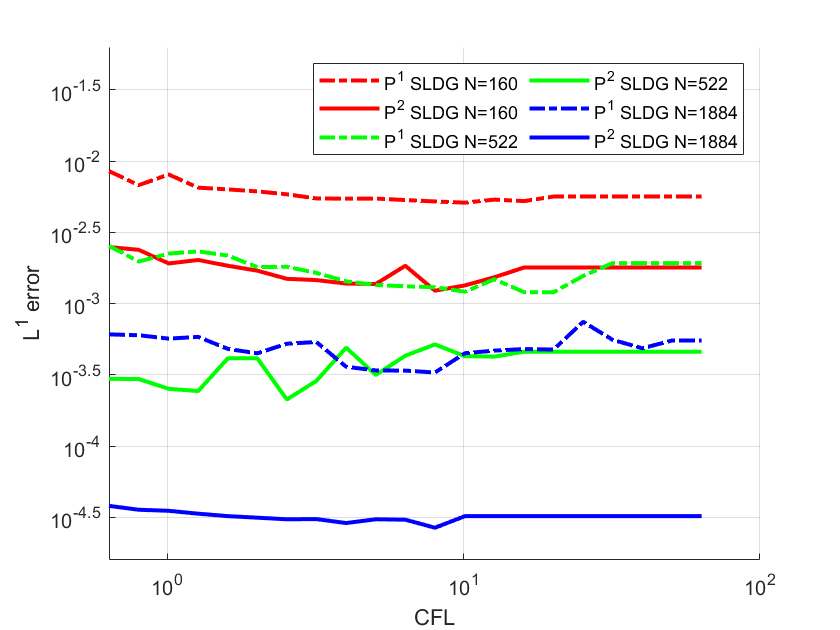}
		\hspace{0in}
		\includegraphics[scale=0.27]{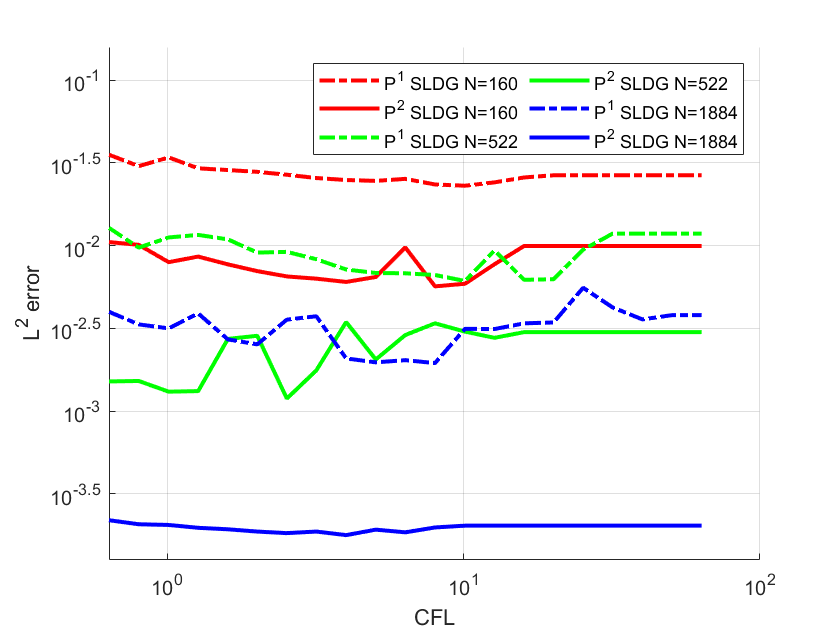}
		\caption{\textbf{Swirling deformation flow.} The $L^1$ and $L^2$ errors versus CFL of $P^1$ and $P^2$ SLDG scheme with $T=1$ on a circle domain with the unstructured meshes of 160, 522 and 1884.}
		\label{fig_swirling_cfl}
	\end{figure}
	
	Moreover, we test the performance of the proposed scheme with swirling deformation flow problem with a slotted disk as a discontinuous initial condition shown in the left top panel of Fig.\ref{fig_swirling_discontinuities}. We implement the experiment with $T=2\pi$ and $CFL=10.5$ on an unstructured mesh of 28996, and the right top panel shows the numerical solution at $t=\frac{\pi}{2}$. The bottom panels of Fig.\ref{fig_swirling_discontinuities} show the numerical solution at $t=\pi$ with a detailed view showing complex structures and the curvilinear upstream mesh. From the figure, we can see that the SLDG scheme on curvilinear unstructured meshes can capture complex structures and greatly reduce the computational complexity. 
	\begin{figure}[htbp]
		\centering
		\includegraphics[scale=0.27]{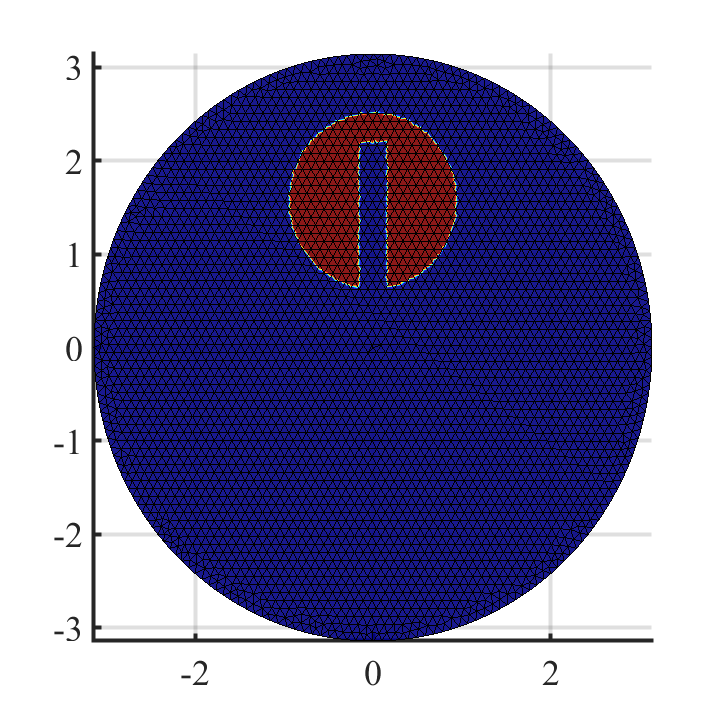}
		\hspace{-0.15in}
		\includegraphics[scale=0.27]{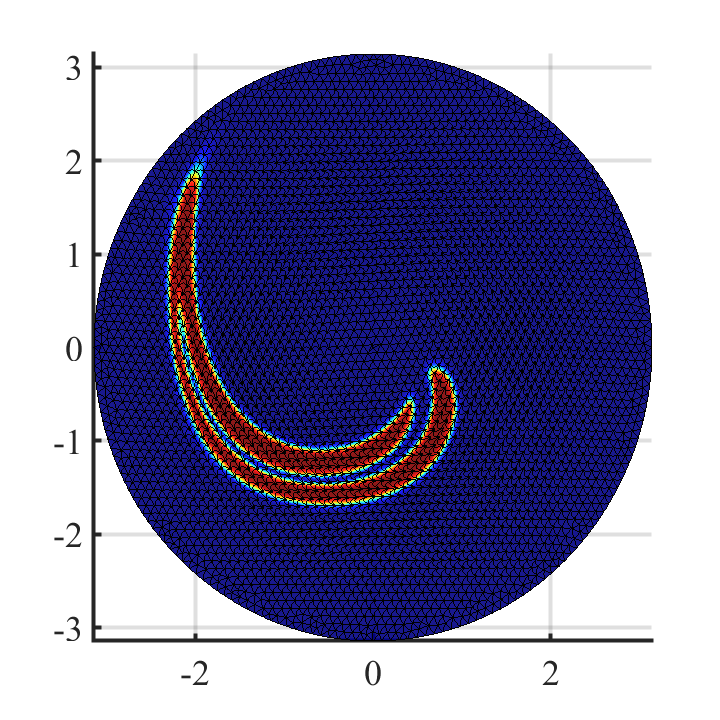}
		\hspace{0in}
		\includegraphics[scale=0.162]{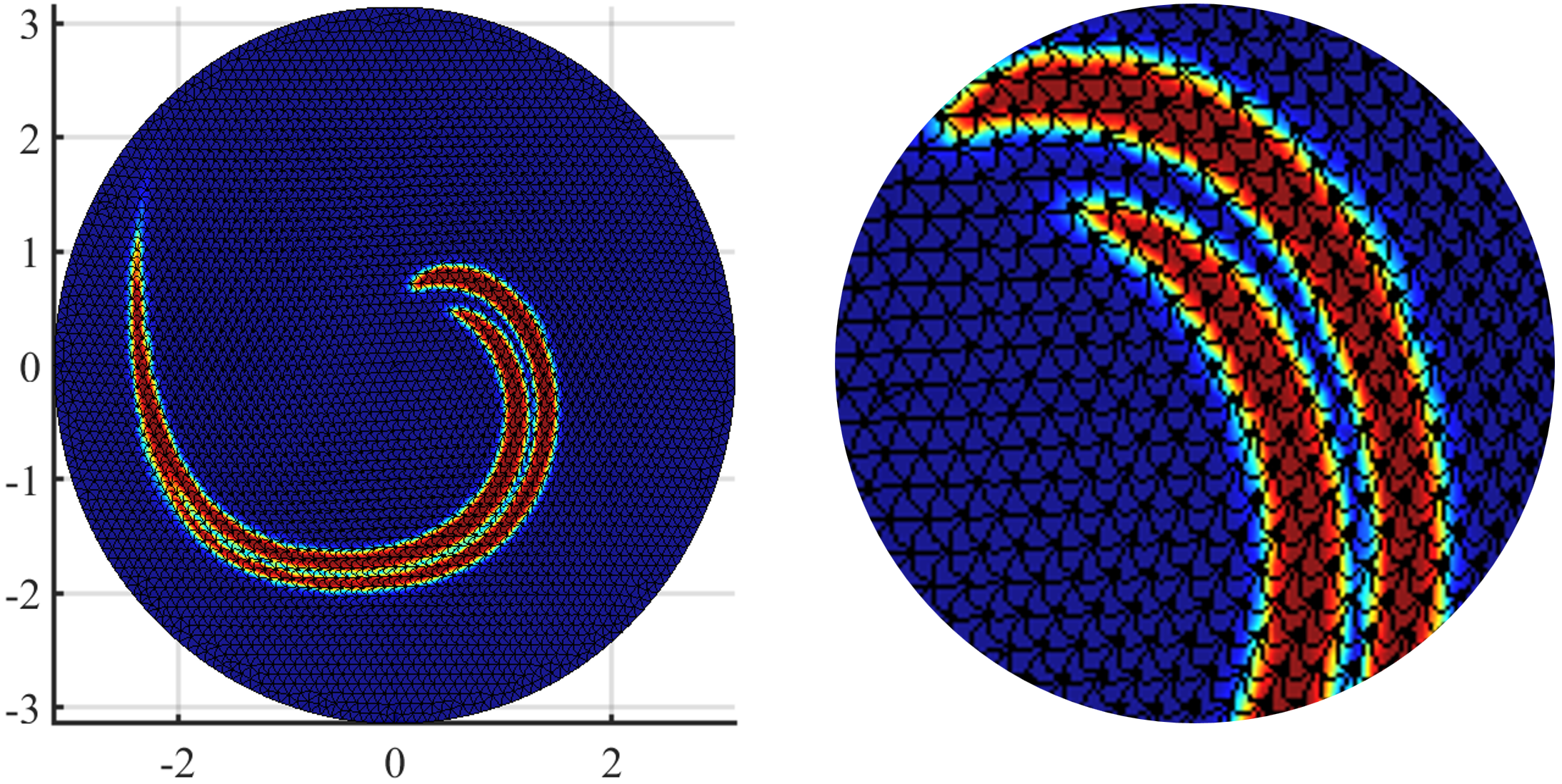}
		\caption{\textbf{Swirling deformation flow.} A slotted disk as initial condition with discontinuties on the left top panel and the numerical solution at $t=\frac{\pi}{2}$ on the right top panel and $t=\pi$ on the left bottom panel. The right bottom panel shows a detailed view of the left bottom panel with curvilinear upstream mesh.}
		\label{fig_swirling_discontinuities} 
	\end{figure}

\begin{example} (Rigid body rotation with discontinuities). 
	We examine numerically solving the rigid body rotation transport equation (\ref*{rigid_body_rotation}) with discontinuous initial condition showed in Fig.\ref*{weno_initial}. The initial data is composed of a cone, a smooth hump along with a slotted disk. 
	\begin{figure}[htbp]
		\centering
		\includegraphics*[scale=0.285]{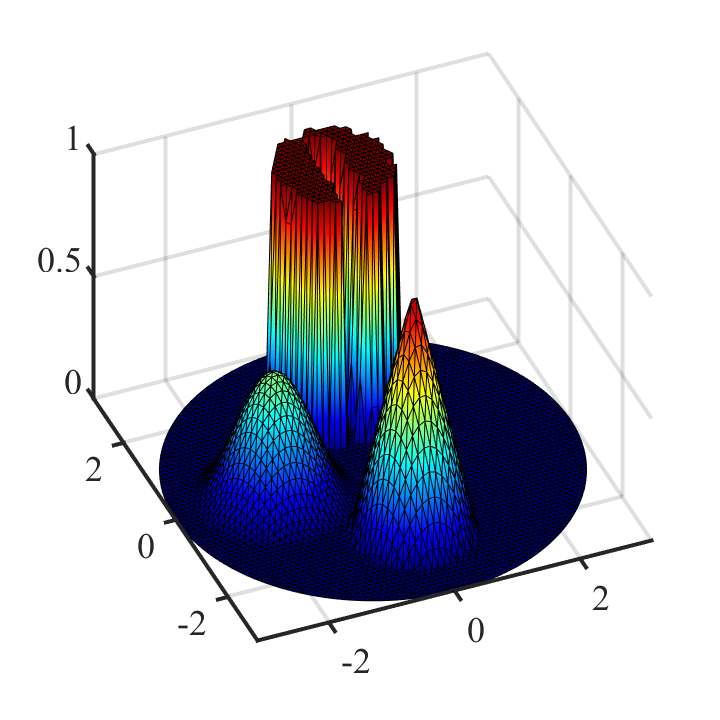}
		\caption{The shape of initial condition with discontinuities which is composed of a slotted disk, a cone and a smooth hump on a circle computational domain with unstructured mesh of 7432 elements.}
		\label{weno_initial}
	\end{figure}
\end{example}

	We apply a circle computational domain to compute the problem \eqref{rigid_body_rotation} solved by RKDG and SLDG schemes and with $k=1,2$ after ten whole rotations, respectively, and the result are plotted in Fig.\ref{weno_circle}. For the RKDG scheme, we set $CFL=0.3$ for $P^1$ and $CFL=0.15$ for $P^2$, for the SLDG scheme, we set $CFL=10^2$ for $P^1$ and $P^2$ schemes.
%	\begin{figure}[H]
%		\centering
%		\includegraphics[scale=0.27]{Figures/weno9(2).png}
%		\hspace{0in}
%		\includegraphics[scale=0.27]{Figures/weno10(2).png}
%		\hspace{0in}
%		\includegraphics[scale=0.24]{Figures/weno9(1).png}
%		\hspace{0in}
%		\includegraphics[scale=0.24]{Figures/weno10(1).png}
%		\caption{\textbf{Rigid body rotation.} The numerical solution with $k=1$(left) and $k=2$(right) of the three shape after a full rotation on square computational domain with unstructured mesh of 7432 elements.}
%		\label{weno_square}
%	\end{figure}

	\begin{figure}[htbp]
		\centering
		\includegraphics[scale=0.27]{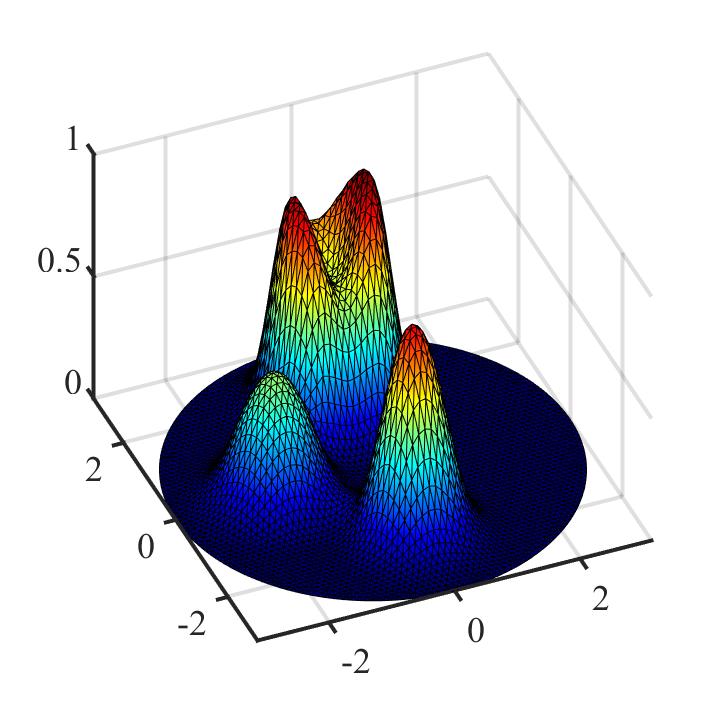}
		\hspace{0in}
		\includegraphics[scale=0.27]{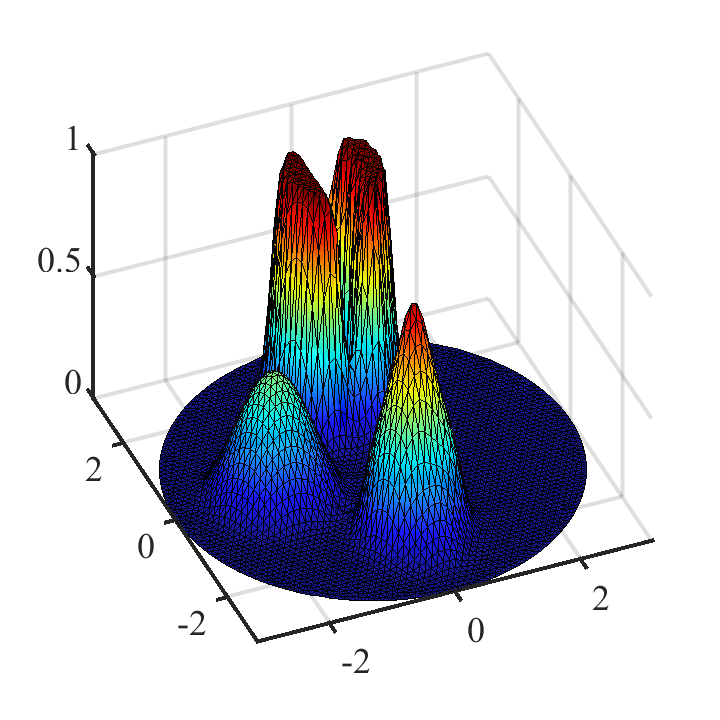}
		\hspace{0in}
		\includegraphics[scale=0.27]{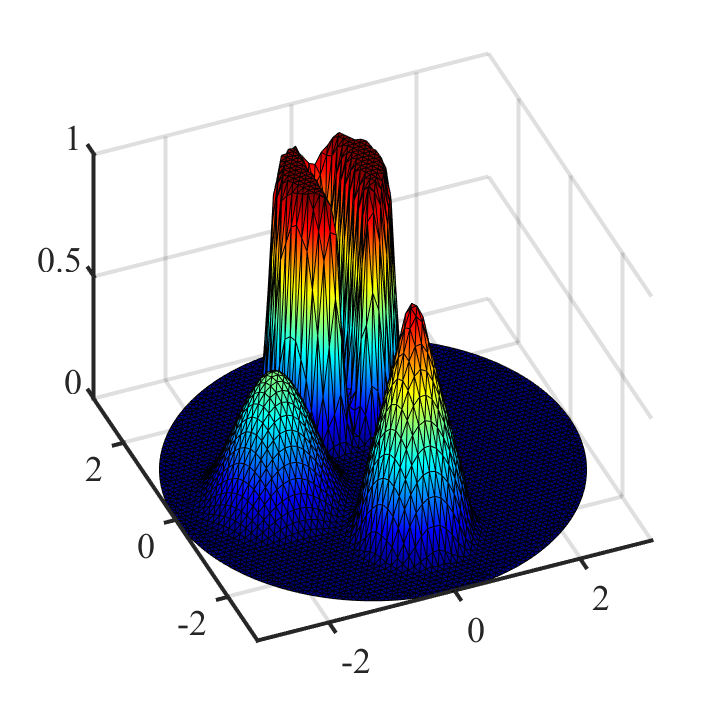}
		\hspace{0in}
		\includegraphics[scale=0.27]{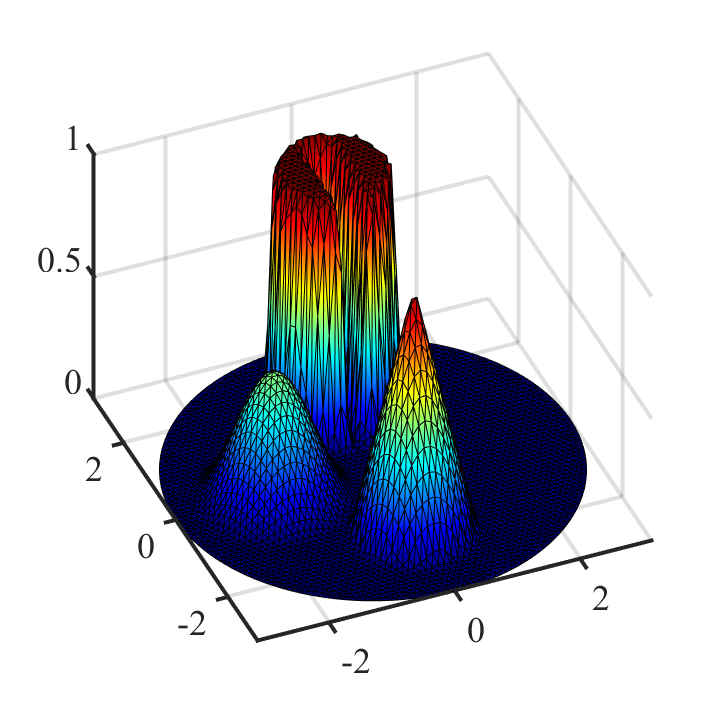}
		\caption{\textbf{Rigid body rotation.} The numerical solution of RKDG with $k=1$(left top) and $k=2$(right top) and of SLDG with $k=1$(left bottom) and $k=2$(right bottom) of the shapes after ten full rotations on circle computational domain with unstructured mesh of 7432 elements.}
		\label{weno_circle}
	\end{figure}
	%多转几圈让左上图看不出来大概形状。
	We can see from Fig.\ref{weno_circle}, these methodologies are capable of managing intricate solution architectures, notwithstanding the emergence of slight oscillations in proximity to discontinuities and by the WENO limiter proposed above is applied, the non-physical oscillations are suppressed. Comparing the result solved with $k=1$ and $k=2$ and the result solved by RKDG and SLDG schemes, we are able to notice that the $P^2$ SLDG scheme can better resolve the solution structures, especially the slotted disk and cone.\\
	Since the velocity field $\boldsymbol{V}(x, y, t)$ in our numerical experiments are divergence free and the slotted disk, cone and smooth hump are all non-negative in the computational domain, then the solution always stays non-negative as time evolves. 
    We applied a PP limiter in Appendix \ref{PP_limiter} in order to preserve positivity of numerical solutions and we compare the results of rigid body rotation with and without the PP limiter, and we mark the elements where functions have negative value as red. We can observe from Fig.\ref{pp1} that once the PP limiter is applied, the positivity of the numerical solution is guaranteed. 
	\begin{figure}[htbp]
		\centering
		\includegraphics[scale=0.27]{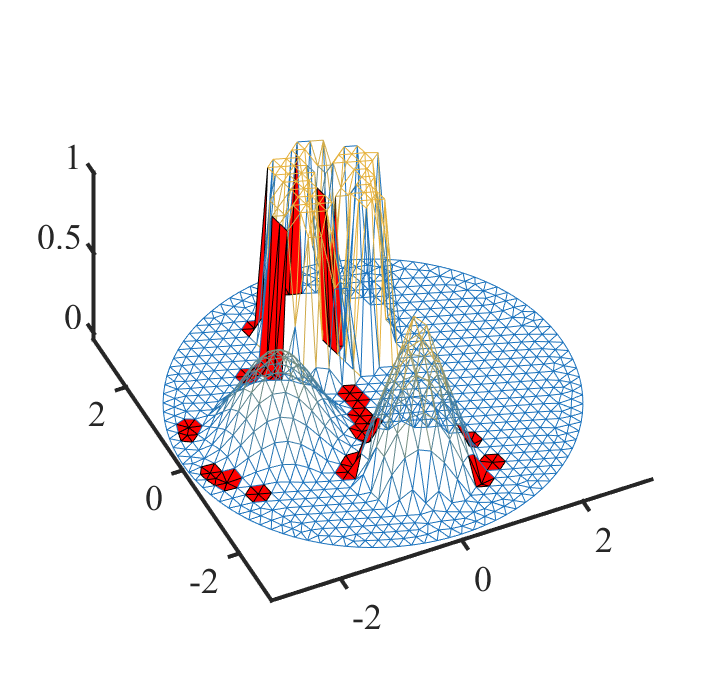}
		\hspace{-0.37in}
		\includegraphics[scale=0.27]{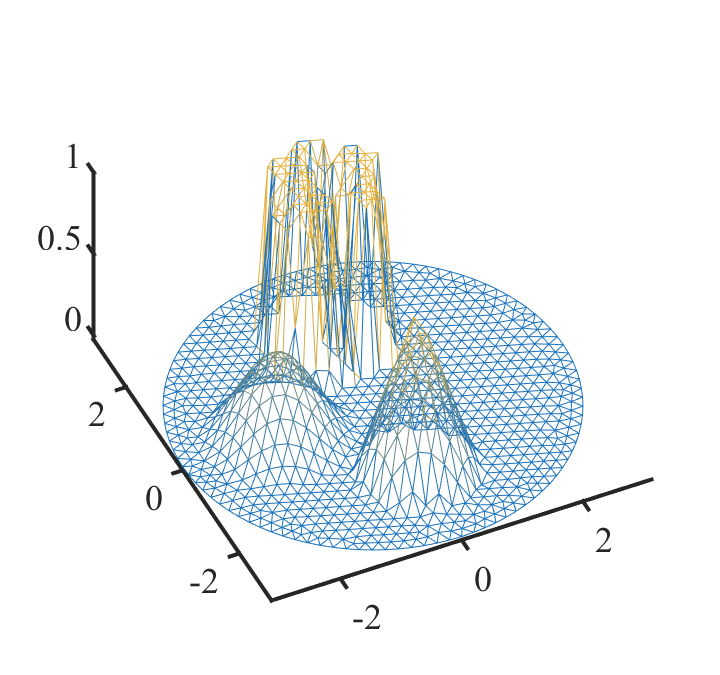}
		\caption{\textbf{Rigid body rotation.} The numerical solution in $P^1$ without PP limiter(left) and with the limiter(right) of the slotted disk, cone and smooth hump after a full rotation on circle computational domain with unstructured mesh of 1884 elements.}
		\label{pp1}
	\end{figure} 
	
	\section{Conclusions} \label{section:conclusion}

We present a high-order conservative SLDG scheme that 
(i) employs  curvilinear upstream elements with  an robust intersection-based remapping algorithm; 
(ii) proves mass conservation and third-order consistency on unstructured meshes;
 and (iii) incorporates WENO and positivity-preserving limiting without degrading conservation and high order accuracy.
  Numerical results show third-order convergence for manufactured and benchmark transports, machine-precision mass conservation, stable large-CFL evolution for the linear problem, and clear accuracy gains of curvilinear over straight-sided upstream elements. 
%The proposed  intersection-based remapping algorithm is very robust by  testing the swirling deformation problems. 
%Thus it can be served as a promising tool for the arbitrary-Lagrangian-Eulerian method. 
	
%	\input{appendix.tex}

%	\input{Introduction.tex}
%	
%	\input{Model.tex}
%	
%	\input{Method.tex}
%	
%	\input{Analysis.tex}
%	
%	\input{Numerics.tex}
%	
%	\input{Conclusion.tex}
%	
%	\input{Appendix.tex}
	
	\bibliographystyle{siamplain}
	\bibliography{ref.bib}
\end{document}

\maketitle

\footnotetext[1]{\funding{The work of the first author was partially supported by the NSFC (No 12201052), National Key Laboratory for Computational Physics (6142A05230201), the Guangdong Provincial Key Laboratory of IRADS (2022B1212010006), Guangdong basic and applied basic research foundation[2025A1515012182].}}
\footnotetext[2]{Research Center of Mathematics, Advanced Institute of Natural Sciences, Beijing Normal University, Zhuhai, 519087, P.R. China. (\email{xfcai@bnu.edu.cn}).}
\footnotetext[3]{Guangdong Provincial/Zhuhai Key Laboratory of Interdisciplinary Research and Application for Data Science, Beijing Normal-Hong Kong Baptist University, Zhuhai, 519087, P.R. China.}
\footnotetext[4]{Institute of Applied Physics and Computational Mathematics and National Key laboratory of Computational physics, and Center for Applied Physics and Technology, Peking University, P.O. Box 8009, Beijing 100088, P.R. China.}
\footnotetext[5]{Hong Kong Baptist University, Kowloon Tong, Hong Kong. (\email{fukunkai@gmail.com}).}
\footnotetext[6]{Institute of Applied Physics and Computational Mathematics and National Key laboratory of Computational physics, P.O. Box 8009, Beijing 100088, P.R. China.}

\section{Positivity preserving limiter}
\label{PP_limiter}\quad
After updating the numerical solution $u^{n+1}|K_j\in V_h^k$, we extend a positivity-preserving limiter \cite{zhang2011maximum} to the high-order reconstructed polynomials. This ensures the results remain positive for certain problems while maintaining both conservation and high-order accuracy. We let $\overline{u}$ is the cell-average of solution $u(x,y)$ on $K_j$ and we modify the polynomials $u(x,y)$ by the following formula 
		\begin{equation}\label{pp}
				\begin{array}{l}
						\hat{u}(x,y)=\theta\left(u(x,y)-\overline{u}\right)+\overline{u}, \\
						\theta=\min \left\{1,\left|\frac{\overline{u}-\epsilon}{\overline{u}-v}\right|\right\}, \quad v=\min\limits_{(x,y)\in K_j} u(x,y),
					\end{array}
			\end{equation}
		where $\epsilon$ is a small number which satisfies $\overline{u}\geqslant\epsilon$, for implemention, we take $\epsilon=10^{-15}$.
		
		%\begin{proof}
		%			Since $v=\min\limits_{(x,y)\in K_j} u(x,y)$ and $\overline{u}\geqslant\epsilon$, there are two cases:\\
		%			If $v\geqslant\epsilon>0$, then $\overline{u}-\epsilon\geqslant\overline{u}-v$, so we have $\theta=1$ and $\hat{u}(x,y)=u(x,y)\geqslant\epsilon$ for all $(x,y)\in K_j$\\
		%			If $v\leqslant\epsilon$, then $\theta=\frac{\overline{u}-\epsilon}{\overline{u}-v}$ and we have
		%			$$\begin{array}{rl}
				%				\hat{u}(x,y)=&\theta u(x,y)+(1-\theta)\overline{u}\\
				%				=&\frac{\overline{u}-\epsilon}{\overline{u}-v}u(x,y)+\frac{\epsilon-v}{\overline{u}-v}\overline{u}\\
				%				\geqslant&\frac{\overline{u}-\epsilon}{\overline{u}-v}v+\frac{\epsilon-v}{\overline{u}-v}\overline{u}\\
				%				=&\epsilon>0, \quad (x,y)\in K_j.
				%			\end{array}$$
		%			So the modified polynomials remains positive and $\frac{1}{|K_j|}\int_{K_j}\hat{u}(x,y)\mathrm{d}x\mathrm{d}y=\overline{u}$.
		%		\end{proof}

For triangular element, the minimum point lies on extreme point or vertices of triangle. We first calculate the extreme point by $
\left\{
\begin{array}{l}
	\partial_x u=0, \\
	\partial_y u=0,
\end{array}
\right.
$ and compare its value with values of vertices to have $v=\min\limits_{(x,y)\in K_j} u(x,y)$.\\
For curvilinear triangular element, similarly, the minimum point lies on extreme point or boundary of element, we can judge if the extreme point inside the curvilinear triangle following the procedures in \ref{vertexes}.

\section{Proof of \eqref{inters}}
\label{appendixA}\quad
We review the definition of $"+"$ and $"-"$ group, and we know that $\left\{  
\begin{array}{l}
	\cup_kA_k^+=A\cup \left(\cup_kA_k^-\right)\\  
	A\cap \left(\cup_kA_k^-\right)=\emptyset\\
	\left(\cup_kA_k^+\right)\cap \left(\cup_kA_k^-\right)=\cup_kA_k^-
\end{array}  
\right.$ for the divided unit A, which is similar to B.\\ 
Therefore, the area of units A and B satisfy $\left\{  
\begin{array}{l}
	\|A\|=\sum\limits_k\|A_k^+\| - \sum\limits_l\|A_l^-\|,\\  
	\|B\|=\sum\limits_k\|B_k^+\| - \sum\limits_l\|B_l^-\|.
\end{array}  
\right.$\\
According to the allocation law of sets
$$ 
\begin{array}{rl}
	&\left(A^-\cap B^+\right)\cap \left(A^+\cap B^-\right)=\left(A^-\cap \left(B\cup B^-\right)\right)\cap \left(\left(A\cup A^-\right)\cap B^-\right)\\
	=&\left[\left(A^-\cap B\right)\cup \left(A^-\cap B^-\right)\right]\cap \left[\left(A\cap B^-\right)\cup \left(A^-\cap B^-\right)\right]\\
	=&\left\{\left[\left(A^-\cap B^-\right)\cap \left(A^-\cap B\right)\right]\cup\left[\left(A^-\cap B^-\right)\cap \left(A^-\cap B^-\right)\right]\right\}\\
	&\cup \left\{\left[\left(A\cap B^-\right)\cap \left(A^-\cap B\right)\right]\cup\left[\left(A\cap B^-\right)\cap \left(A^-\cap B^-\right)\right]\right\}\\
	=&\left(A\cap A^-\cap B\cap B^-\right)\cup \left(A\cap A^-\cap B^-\right)\cup \left(B\cap B^-\cap A^-\right)\cup \left(A^-\cap B^-\right)\\
	=&\left(\emptyset\cap \emptyset\right)\cup \left(\emptyset\cap B^-\right)\cup \left(\emptyset\cap A^-\right)\cup \left(A^-\cap B^-\right)\\
	=&A^-\cap B^-.
\end{array}$$
%	$$ 
%	\left(A^-\cap B^+\right)\cap \left(A^+\cap B^-\right)=A^-\cap B^-.
%	$$
By applying this conclusion to each partition in \eqref{inters}, the theorem can be proved.

\section{Proof of Lemmas}\label{appendixB}
\subsection{Proof of Lemma \ref{lemma_derivative}}\label{proof_lemma_derivative}\quad
Taking the derivative of both side of equation \eqref{parametric_equation} with respect to $\theta$, since $t$ is independent of parameter $\theta$, we can interchange the order of integration and differentiation, so we have
\begin{equation}\label{derivative2}
	\left[\begin{matrix}
		\frac{\mathrm{d} \mathcal{X}(\theta)}{\mathrm{d} \theta} \\
		\frac{\mathrm{d} \mathcal{Y}(\theta)}{\mathrm{d} \theta} \\
	\end{matrix} \right]=\left[\begin{matrix}
		\frac{\mathrm{d} x(\theta)}{\mathrm{d} \theta} \\
		\frac{\mathrm{d} y(\theta)}{\mathrm{d} \theta} \\
	\end{matrix} \right]+\left[\begin{matrix}
		\int_{t^{n+1}}^{t^n}\frac{\mathrm{d}}{\mathrm{d} \theta}a\Big(X\big(x(\theta),y(\theta);t\big),Y\big(x(\theta),y(\theta);t\big),t\Big)dt \\
		\int_{t^{n+1}}^{t^n}\frac{\mathrm{d}}{\mathrm{d} \theta}b\Big(X\big(x(\theta),y(\theta);t\big),Y\big(x(\theta),y(\theta);t\big),t\Big)dt \\
	\end{matrix} \right].
\end{equation}
From the parametric equation of $\overline{v_1v_2}$, the first term of the R.H.S of \eqref{derivative2} is $\left[\begin{matrix}
	\Delta x \\
	\Delta y \\
\end{matrix} \right]$, the total derivative of $a(X,Y,t)$ respect to $\theta$ is 
\begin{equation}\label{total_derivative}
	\frac{\mathrm{d}a\left(X,Y,t\right)}{\mathrm{d} \theta}=\frac{\partial a\left(X,Y,t\right)}{\partial X}\frac{\mathrm{d} X}{\mathrm{d} \theta}+\frac{\partial a\left(X,Y,t\right)}{\partial Y}\frac{\mathrm{d} Y}{\mathrm{d} \theta}+\frac{\partial a\left(X,Y,t\right)}{\partial t}\frac{\mathrm{d} t}{\mathrm{d} \theta}.
\end{equation}
The total derivative of $b(X,Y,t)$ respect to $\theta$ is similar to equation \eqref{total_derivative}. Since $t$ is independent of parameter $\theta$, $\frac{\partial t}{\partial \theta}=0$. For simplicity, we let $$\left[\begin{matrix}
	a'_1&a'_2 \\
	b'_1&b'_2 \\
\end{matrix} \right]=\left[\begin{matrix}
	\frac{\partial a\left(X,Y,t\right)}{\partial X}&\frac{\partial a\left(X,Y,t\right)}{\partial Y} \\
	\frac{\partial b\left(X,Y,t\right)}{\partial X}&\frac{\partial b\left(X,Y,t\right)}{\partial Y} \\
\end{matrix} \right].$$
For explicitly given velocity function $a$ and $b$, we might able to have $\left[\begin{matrix}
	a'_1&a'_2 \\
	b'_1&b'_2 \\
\end{matrix} \right]$ explicitly. But $X\big(x(\theta),y(\theta);t\big)$ and $Y\big(x(\theta),y(\theta);t\big)$ are obtained by solving the ordinary differential equation \eqref{characteristic}, making them difficult to be expressed explicitly, thus hard to calculate the integral over $[t^{n+1},t^n]$. To obtain that, since $a(x,y,t)\in C^\infty$, we perform a Taylor expansion of $\frac{\mathrm{d}a\left(X,Y,t\right)}{\mathrm{d} \theta}$ at $\big(x(\theta),y(\theta);t^n\big)$,
\begin{equation}
	\frac{\mathrm{d}a\left(X,Y,t\right)}{\mathrm{d} \theta}=\frac{\mathrm{d}a\left(X,Y,t^n\right)}{\mathrm{d} \theta}+(t-t^n)\frac{\mathrm{d}^2a\left(X,Y,t^n\right)}{\mathrm{d} \theta \mathrm{d}t}+O(\Delta t^2).
\end{equation}
So the integral of $\frac{\mathrm{d}a}{\mathrm{d}\theta}$ over $[t^{n+1},t^n]$ is $$\int_{t^{n+1}}^{t^n}\frac{\mathrm{d}a}{\mathrm{d} \theta}dt=(t^n-t^{n+1})\frac{\mathrm{d}a\left(X,Y,t^n\right)}{\mathrm{d} \theta}+\frac{(t^{n+1}-t^n)^2}{2}\frac{\mathrm{d}^2a\left(X,Y,t^n\right)}{\mathrm{d} \theta \mathrm{d}t}+O(\Delta t^3).$$
By the definition of $\mathcal{X}$, we have $\frac{\mathrm{d}a\left(X,Y,t^n\right)}{\mathrm{d} \theta}=a'_1|_{\big(x(\theta),y(\theta);t^n\big)}\frac{\mathrm{d} \mathcal{X}}{\mathrm{d} \theta}+a'_2|_{\big(x(\theta),y(\theta);t^n\big)}\frac{\mathrm{d} \mathcal{Y}}{\mathrm{d} \theta}$.\\
We have the similar result of $b(X,Y,t)$, combine them with the matrix form and drop the subscript $\big(x(\theta),y(\theta);t^n\big)$ of $\left[\begin{matrix}
	a'_1&a'_2 \\
	b'_1&b'_2 \\
\end{matrix} \right]$ for simplicity, we have
\begin{equation}\label{derivative3}
	\left[\begin{matrix}
		\frac{\mathrm{d} \mathcal{X}(\theta)}{\mathrm{d} \theta} \\
		\frac{\mathrm{d} \mathcal{Y}(\theta)}{\mathrm{d} \theta} \\
	\end{matrix} \right]=\left[\begin{matrix}
		\Delta x \\
		\Delta y \\
	\end{matrix} \right]-\Delta t\left[\begin{matrix}
		a'_1&a'_2 \\
		b'_1&b'_2 \\
	\end{matrix} \right]\left[\begin{matrix}
		\frac{\mathrm{d} \mathcal{X}(\theta)}{\mathrm{d} \theta} \\
		\frac{\mathrm{d} \mathcal{Y}(\theta)}{\mathrm{d} \theta} \\
	\end{matrix} \right]+O(\Delta t^2)\left[\begin{matrix}
		\frac{\mathrm{d}^2a\left(X,Y,t^n\right)}{\mathrm{d} \theta \mathrm{d}t} \\
		\frac{\mathrm{d}^2b\left(X,Y,t^n\right)}{\mathrm{d} \theta \mathrm{d}t} \\
	\end{matrix} \right].
\end{equation}
Since $a(x,y,t),b(x,y,t)\in C^{\infty}$, $\left\lVert\begin{matrix}
	\frac{\mathrm{d}^2a\left(X,Y,t^n\right)}{\mathrm{d} \theta \mathrm{d}t} \\
	\frac{\mathrm{d}^2b\left(X,Y,t^n\right)}{\mathrm{d} \theta \mathrm{d}t} \\
\end{matrix} \right\rVert$ is bounded, so we can drop the high-order term in \eqref{derivative3}, we have
\begin{equation}\label{derivative4}
	\left[\begin{matrix}
		1+\Delta ta'_1&\Delta ta'_2 \\
		\Delta tb'_1&1+\Delta tb'_2 \\
	\end{matrix} \right]\left[\begin{matrix}
		\frac{\mathrm{d} \mathcal{X}(\theta)}{\mathrm{d} \theta} \\
		\frac{\mathrm{d} \mathcal{Y}(\theta)}{\mathrm{d} \theta} \\
	\end{matrix} \right]=\left[\begin{matrix}
		\Delta x \\
		\Delta y \\
	\end{matrix} \right].
\end{equation}
Denote the matrix in \eqref{derivative4} as A, and it is easy to compute the determinant of $A$ that $\det(A)=1+O(\Delta t)$. For a sufficiently small $\Delta t$, $\det(A)\neq0$, so we have the inverse matrix $A^{-1}=\frac{1}{\det(A)}\left[\begin{matrix}
	1+\Delta tb'_2&-\Delta ta'_2 \\
	-\Delta tb'_1&1+\Delta ta'_1 \\
\end{matrix} \right]$. $\Delta x$ and $\Delta y$ are usually small positive numbers for obtaining good approximations, hence we obtain
\begin{equation}
	\left[\begin{matrix}
		\frac{\mathrm{d} \mathcal{X}(\theta)}{\mathrm{d} \theta} \\
		\frac{\mathrm{d} \mathcal{Y}(\theta)}{\mathrm{d} \theta} \\
	\end{matrix} \right]=\frac{1}{\det(A)}\left[\begin{matrix}
		\Delta x+\Delta t\Big(b'_2\Delta x-a'_2\Delta y\Big)\\
		\Delta y+\Delta t\Big(a'_1\Delta y-b'_1\Delta x\Big)\\
	\end{matrix} \right]=\left[\begin{matrix}
		\Delta x+o(\Delta t)\\
		\Delta y+o(\Delta t)\\
	\end{matrix} \right].
\end{equation}

\subsection{Proof of Lemma \ref{lemma_derivative2}}\label{proof_lemma_derivative2}\quad
Similar to the proof of Lemma \ref{lemma_derivative}, taking the k-order derivative of both sides of \eqref{parametric_equation} with respect to $\theta$, we have
\begin{equation}\label{derivative5}
	\left[\begin{matrix}
		\frac{\mathrm{d}^k \mathcal{X}(\theta)}{\mathrm{d} \theta^k} \\
		\frac{\mathrm{d}^k \mathcal{Y}(\theta)}{\mathrm{d} \theta^k} \\
	\end{matrix} \right]=\left[\begin{matrix}
		\frac{\mathrm{d}^k x(\theta)}{\mathrm{d} \theta^k} \\
		\frac{\mathrm{d}^k y(\theta)}{\mathrm{d} \theta^k} \\
	\end{matrix} \right]+\left[\begin{matrix}
		\int_{t^{n+1}}^{t^n}\frac{\mathrm{d}^k}{\mathrm{d} \theta^k}a\Big(X\big(x(\theta),y(\theta);t\big),Y\big(x(\theta),y(\theta);t\big),t\Big)dt \\
		\int_{t^{n+1}}^{t^n}\frac{\mathrm{d}^k}{\mathrm{d} \theta^k}b\Big(X\big(x(\theta),y(\theta);t\big),Y\big(x(\theta),y(\theta);t\big),t\Big)dt \\
	\end{matrix} \right].
\end{equation}
From the parametric equation of $\overline{v_1v_2}$, the first term of the right-hand side of \eqref{derivative5} is $\left[\begin{matrix}
	0 \\
	0 \\
\end{matrix} \right]$.\\
We first consider the second-order total derivative of $a(X,Y,t)$ respect to $\theta$,
\begin{equation}\label{total_derivative2}
	\frac{\mathrm{d}^2a}{\mathrm{d} \theta^2}=\frac{\partial^2 a}{\partial X^2}\left(\frac{\mathrm{d}X}{\mathrm{d}\theta}\right)^2+2\frac{\partial^2 a}{\partial X\partial Y}\frac{\mathrm{d}X}{\mathrm{d}\theta}\frac{\mathrm{d}Y}{\mathrm{d}\theta}+\frac{\partial^2 a}{\partial Y^2}\left(\frac{\mathrm{d}Y}{\mathrm{d}\theta}\right)^2+\frac{\partial a}{\partial X}\frac{\mathrm{d}^2X}{\mathrm{d}\theta^2}+\frac{\partial a}{\partial Y}\frac{\mathrm{d}^2Y}{\mathrm{d}\theta^2}.
\end{equation}
We still perform a Taylor expansion of $\frac{\mathrm{d}^2a\left(X,Y,t\right)}{\mathrm{d} \theta^2}$ at $\big(x(\theta),y(\theta);t^n\big)$ and drop the high-order term, we have
$$
\frac{\mathrm{d}^2a\left(X,Y,t\right)}{\mathrm{d} \theta^2}=\frac{\mathrm{d}^2a\left(X,Y,t^n\right)}{\mathrm{d} \theta^2}
$$
So the integral of $\frac{\mathrm{d}^2a}{\mathrm{d}\theta^2}$ over $[t^{n+1},t^n]$ is $$\int_{t^{n+1}}^{t^n}\frac{\mathrm{d}^2a}{\mathrm{d} \theta^2}dt=(t^n-t^{n+1})\frac{\mathrm{d}^2a\left(X,Y,t^n\right)}{\mathrm{d} \theta^2}.$$
By the definition of $\mathcal{X}$, we have
\begin{equation}\label{derivative6}
	\frac{\mathrm{d}^2a\left(X,Y,t^n\right)}{\mathrm{d} \theta^2}=a'_{11}\left(\frac{\mathrm{d}\mathcal{X}}{\mathrm{d}\theta}\right)^2+2a'_{12}\frac{\mathrm{d}\mathcal{X}}{\mathrm{d}\theta}\frac{\mathrm{d}\mathcal{Y}}{\mathrm{d}\theta}+a'_{22}\left(\frac{\mathrm{d}\mathcal{Y}}{\mathrm{d}\theta}\right)^2+a'_1\frac{\mathrm{d}^2\mathcal{X}}{\mathrm{d}\theta^2}+a'_2\frac{\mathrm{d}^2\mathrm{Y}}{\mathrm{d}\theta^2}.
\end{equation}
Since $\Delta t\sim\Delta x\sim\Delta y$, by applying the conclusion of Lemma \ref{lemma_derivative}, the first three terms of \eqref{derivative6} is of order $O(\Delta t^2)$.\\
We have the similar result of $b(X,Y,t)$, combine them with the matrix form and we have
\begin{equation}\label{derivative7}
	\left[\begin{matrix}
		\frac{\mathrm{d}^2 \mathcal{X}(\theta)}{\mathrm{d} \theta^2} \\
		\frac{\mathrm{d}^2 \mathcal{Y}(\theta)}{\mathrm{d} \theta^2} \\
	\end{matrix} \right]=-\Delta t\left[\begin{matrix}
		a'_1&a'_2 \\
		b'_1&b'_2 \\
	\end{matrix} \right]\left[\begin{matrix}
		\frac{\mathrm{d}^2 \mathcal{X}(\theta)}{\mathrm{d} \theta^2} \\
		\frac{\mathrm{d}^2 \mathcal{Y}(\theta)}{\mathrm{d} \theta^2} \\
	\end{matrix} \right]+O(\Delta t^2)\delta A,
\end{equation}
where $\delta A\in\mathbb{R}^2$ and $\lVert\delta A\rVert$ is bounded, by rearranging \eqref{derivative7}, we obtain
\begin{equation}\label{derivative8}
	\left[\begin{matrix}
		1+\Delta ta'_1&\Delta ta'_2 \\
		\Delta tb'_1&1+\Delta tb'_2 \\
	\end{matrix} \right]\left[\begin{matrix}
		\frac{\mathrm{d}^2 \mathcal{X}(\theta)}{\mathrm{d} \theta^2} \\
		\frac{\mathrm{d}^2 \mathcal{Y}(\theta)}{\mathrm{d} \theta^2} \\
	\end{matrix} \right]=O(\Delta t^2)\delta A.
\end{equation}
The matrix in equation \eqref{derivative8} is same as matrix in \eqref{derivative4}, we already have the inverse of $A$. Hence, we obtain
\begin{equation}\label{derivative9}
	\left[\begin{matrix}
		\frac{\mathrm{d}^2 \mathcal{X}(\theta)}{\mathrm{d} \theta^2} \\
		\frac{\mathrm{d}^2 \mathcal{Y}(\theta)}{\mathrm{d} \theta^2} \\
	\end{matrix} \right]=\frac{O(\Delta t^2)}{\det(A)}\left[\begin{matrix}
		1+\Delta t\Big(b'_2-a'_2\Big)\\
		1+\Delta t\Big(a'_1-b'_1\Big)\\
	\end{matrix} \right]\delta A=\left[\begin{matrix}
		O(\Delta t^2)\\
		O(\Delta t^2)\\
	\end{matrix} \right].
\end{equation}
By inductive method, the k-order derivative of $\mathcal{X}(\theta)$ and $\mathcal{Y}(\theta)$ can be obtain from the conclusion of lower order derivative, so it is proved in the same manner.

\subsection{Proof of Lemma \ref{lemma_distance}}\label{proof_lemma_distance}\quad
It can easily be proved by applying Lemma \ref{lemma_derivative}. Integrate both sides of \eqref{derivative1} with respect to $\theta$ over $[0,1]$, we have
\begin{equation}\label{upstream1}
	x^\star_2-x^\star_1=\int_{0}^{1}\frac{\mathrm{d}\mathcal{X}(\theta)}{\mathrm{d}\theta}\mathrm{d}\theta=\int_{0}^{1}\Delta x+o(\Delta t)\mathrm{d}\theta=\Delta x+o(\Delta t).
\end{equation}
Since $h\sim\Delta t\sim\Delta x$, we can express $\Delta x+o(\Delta t)$ as $O(h)$. The proof of $y^\star_2-y^\star_1$ is the same as \eqref{upstream1}. The distance between $v^\star_1$ and $v^\star_2$ is
$$d(v^\star_1,v^\star_2)=\sqrt{(x^\star_2-x^\star_1)^2+(y^\star_2-y^\star_1)^2}=\sqrt{O(h^2)}=O(h).$$
The proof of Lemma \ref{lemma_distance} is completed.

\subsection{Proof of Lemma \ref{projection_lemma}}\label{proof_lemma_projection}\quad
By \eqref{projection_vector}, we know  $Pu-u\perp v$ for any $v\in V_h^k$, therefore $Pu-u\perp V_h^k$. From Theorem \ref{projection_theorem}, we have 
\begin{equation}\label{projection1}
	\left\lVert Pu-u\right\rVert\leqslant\left\lVert v-u\right\rVert \text{ for any }v\in V_h^k. 
\end{equation}
Then we consider the Taylor expansion of $u(x,y)$ at center of $K_j$, namely $(x_0,y_0)$,
%\begin{equation}\label{taylor}
%\begin{aligned}
%	u(x,y)&=P_n(x,y)+R_n(x,y)=u(x_0,y_0)+u_x(x_0,y_0)(x-x_0)+u_y(x_0,y_0)(y-y_0)\\
%	+&\frac{1}{2!}[ u_{xx}(x_0,y_0)(x-x_0)^2+2u_{xy}(x_0,y_0)(x-x_0)(y-y_0)+u_{yy}(x_0,y_0)(y-y_0)^2 ]\\
%	+&\frac{1}{3!}[ u_{xxx}(x_0,y_0)(x-x_0)^3+3u_{xxy}(x_0,y_0)(x-x_0)^2(y-y_0)\\
%	&+u_{xyy}(x_0,y_0)(x-x_0)(y-y_0)^2+u_{yyy}(x_0,y_0)(y-y_0)^3 ]+\cdot\cdot\cdot+R_n(x,y),
%\end{aligned}
%\end{equation}
\begin{equation*}\label{taylor}
	\begin{aligned}
		u(x,y)&=P_n(x,y)+R_n(x,y)=u(x_0,y_0)+u_x(x_0,y_0)(x-x_0)+u_y(x_0,y_0)(y-y_0)\\
		& +\cdot\cdot\cdot+R_n(x,y),
	\end{aligned}
\end{equation*}
where $R_n(x,y)$ is the remainder term in the n-order Taylor expansion. We express $R_n(x,y)$ in the form of Lagrange remainder:
\begin{equation}\label{remainder1}
	R_n(x,y)=\sum\limits_{\left|\boldsymbol{\alpha}\right|=n+1}\frac{1}{\boldsymbol{\alpha}!}D^{\boldsymbol{\alpha}}u(\xi,\eta)(x-x_0)^{\alpha_1}(y-y_0)^{\alpha_2},
\end{equation}
where $\boldsymbol{\alpha}=(\alpha_1,\alpha_2)$ is double index, $D^{\boldsymbol{\alpha}}$ is higher order derivative and $(\xi,\eta)$ is an interior point of $K_j$. We consider $m(x,y)=P_k(x,y)\in V_h^k$, then we have
\begin{equation}\label{remainder2}
	\left\lVert m(x,y)-u(x,y) \right\rVert=\left\lVert R_k(x,y) \right\rVert.
\end{equation} 
Since $(x,y)\in K_j$ and $h\sim\Delta x\sim\Delta y$, we have$(x-x_0)^{\alpha_1}(y-y_0)^{\alpha_2}\leqslant h^{k+1}$. Let $C=\left(n+2\right)\max\limits_{\left|\boldsymbol{\alpha}\right|=n+1}\frac{1}{\boldsymbol{\alpha}!}D^{\boldsymbol{\alpha}}u(\xi,\eta)$, and we substitute these into \eqref{remainder1}, equation \eqref{remainder2} becomes $\left\lVert m(x,y)-u(x,y) \right\rVert\leqslant Ch^{k+1}$. From the arbitrary of $v$ in \eqref{projection1}, we have $\left\lVert Pu-u\right\rVert\leqslant\left\lVert m-u\right\rVert\leqslant Ch^{k+1}$.

\subsection{Proof of Lemma \ref{edge_approximate}}\label{proof_lemma_edge}\quad
Recall the isoparametric map \eqref{shapefunction2}, we know that there exist $k+1$ intersection points between $\widehat{v^\star_1v^\star_2}$ and $\overline{v^\star_1v^\star_2}$, we can define them as $(x_i,y_i)$ with parameters $\theta_i$, $i=0,\cdot\cdot\cdot,k$.\\
From Lemma \ref{lemma_derivative} and \ref{lemma_derivative2}, we know the parametric equations \eqref{parametric_equation} of $\widehat{v^\star_1v^\star_2}$ is $(k+1)$-times continuously differentiable, i.e. $\mathcal{X}(\theta),\mathcal{Y}(\theta)\in C^{k+1}([0,1])$.\\
We define the distance between $\mathcal{X}$ and $x^F$ as $E_x^k(\theta)=\mathcal{X}(\theta)-x^F(\theta)$, and we want to show that 
\begin{equation}\label{E_x^k}
	E_x^k(\theta)=\frac{\mathcal{X}^{(k+1)}(\xi)}{(k+1)!}\prod\limits_{i=0}^{k}(\theta-\theta_i)\Delta x.
\end{equation}
We define a function
\begin{equation}\label{G(lambda)}
	G(\lambda)=E_x^k(\lambda)-\frac{E_x^k(\theta)}{\prod\limits_{i=0}^{k}(\theta-\theta_i)\Delta x}\prod_{i=0}^{k}(\lambda-\theta_i)\Delta x=0,\ \ \theta\neq\theta_i.
\end{equation}
When $\lambda=\theta_i,\ \ i=1,\cdot\cdot\cdot,k$, we have $E_x^k(\theta_i)=\mathcal{X}(\theta_i)-x^F(\theta_i)=0$ and $\prod_{i=0}^{k}(\lambda-\theta_i)=0$, therefore $G(\theta_i)=0,\ \ i=1,\cdot\cdot\cdot,k$. Besides, when $\lambda=\theta$, we have $G(\theta)=0$, so $G(\lambda)=0$ has at least $(k+2)$ distinct roots.\\ 
By Rolle's Theorem, we know $G'(\lambda)=0$ has at least $(k+1)$ roots, and they lie between the roots of $G(\lambda)=0$. By induction, $G^{(k+1)}(\lambda)=0$ has at least one root $\xi$ and it lies inside $(0,1)$.\\
$\prod_{i=0}^{k}(\lambda-\theta_i)$, a monic polynomials of degree $(k+1)$, its $(k+1)-$derivative is equal to $(k+1)!$ and $x^F(\lambda)\in P^k$, its $(k+1)-$derivative is equal to $0$. Therefore 
\begin{equation}\label{G^k+1}
	G^{(k+1)}(\xi)=\mathcal{X}^{(k+1)}(\xi)-\frac{(k+1)!}{\prod\limits_{i=0}^{k}(\theta-\theta_i)\Delta x}E_x^k(\theta)=0.
\end{equation}
By rearranging \eqref{G^k+1}, we obtain \eqref{E_x^k}. Since $\theta,\theta_i\in[0,1]$ and $h\sim\Delta x$, we have $\prod\limits_{i=0}^{k}(\theta-\theta_i)\Delta x\leqslant h^{k+1}$. Let $C_1=\frac{\mathcal{X}^{(k+1)}(\xi)}{(k+1)!}$, then we have $$\left\lVert\mathcal{X}(\theta)-x^F(\theta)\right\rVert=\left\lVert E_x^k(\theta)\right\rVert\leqslant C_1h^{k+1}.$$
The proof of $\left\lVert\mathcal{Y}(\theta)-y^F(\theta)\right\rVert$ is similar with $\left\lVert\mathcal{X}(\theta)-x^F(\theta)\right\rVert$, so the equation \eqref{xy_h^k+1} is proved.\\
Let $C=\sqrt{C_1^2+C_2^2}$, then we have $$\int_{0}^{1}\sqrt{\left(\mathcal{X}-x^F\right)^2+\left(\mathcal{Y}-y^F\right)^2}\mathrm{d}\theta\leqslant\int_{0}^{1}\sqrt{(C_1^2+C_2^2)h^{2k+2}}\mathrm{d}\theta=Ch^{k+1},$$
so the equation \eqref{d_h^k+1} is proved.

\section{Proof of Theorem \ref{consistency}}\label{appendixC}\quad
Having the lemmas, we are now in a position to make an error analysis. We assume $\overline{u}^{n+1}(x,y)$ is the discontinuous Galerkin solution computed by exact upstream element. Then we can divide the norm of error into two pieces:
$$\begin{array}{rl}
	\left\lVert u^{n+1}-u(t^{n+1}) \right\rVert&=\left\lVert u^{n+1}-\overline{u}^{n+1}+\overline{u}^{n+1}-u(t^{n+1}) \right\rVert\\
	&\leqslant \left\lVert u^{n+1}-\overline{u}^{n+1} \right\rVert+\left\lVert \overline{u}^{n+1}-u(t^{n+1}) \right\rVert.
\end{array}$$
From Lemma \ref{projection_lemma} and equation \eqref{sl}, we know that for $\forall\Psi\in P^2(K_j)$, we have $$(\overline{u}^{n+1}-u(t^{n+1}),\Psi)=\int_{K_j}\left(\overline{u}^{n+1}-u(t^{n+1})\right)\Psi\mathrm{d}x=0.$$
Therefore, $\overline{u}$ is the projection of $u(t^{n+1})$ on $V_h^2$, then there exists a constant $C_1$, such that
\begin{equation}\label{u(t^n+1)}
	\left\lVert \overline{u}^{n+1}-u(t^{n+1}) \right\rVert_{L^2}\leqslant C_1h^3.
\end{equation}
The case in $L^1$ norm is similar with equation \eqref{u(t^n+1)}. From Lemma \ref{edge_approximate}, we know that there exists a constant $\tilde{C}_2^{(1)}$, such that for the area between $\widehat{v^\star_1v^\star_2}$ and $\overline{v^\star_1v^\star_2}$, we have $d(\widehat{v^\star_1v^\star_2},\overline{v^\star_1v^\star_2})\leqslant \tilde{C}_2^{(1)}h^3$, similar with it, $d(\widehat{v^\star_2v^\star_3},\overline{v^\star_2v^\star_3})\leqslant \tilde{C}_2^{(2)}h^3$ and $d(\widehat{v^\star_3v^\star_1},\overline{v^\star_3v^\star_1})\leqslant \tilde{C}_2^{(3)}h^3$. Let $\Delta K_j^\star$ denotes the domain between exact and approximated upstream element, we let $\tilde{C}_2=\sum\limits_{i=1}^3\tilde{C}_2^{(i)}$, then we have
\begin{equation}
	|\Delta K_j^\star|:=\int_{\Delta K_j^\star}1\mathrm{d}\bm{\theta}=\tilde{C}_2h^3,
\end{equation}
where $\mathrm{d}\bm{\theta}:=\frac{1}{h^2}\mathrm{d}x\mathrm{d}y$, since in the reference element, we have $\mathrm{d}x=\Delta x\mathrm{d}\theta$, $\mathrm{d}y=\Delta y\mathrm{d}\theta$.\\
From \eqref{sl} and the definition of $\overline{u}^{n+1}$ and $\Delta K_j^\star$, we have 
\begin{equation}\label{delta_sl}
	\int_{K_j}\left|u^{n+1}-\overline{u}^{n+1}\right|\Psi \mathrm{d}x\mathrm{d}y=\int_{\Delta K_j^\star}u^n\psi\mathrm{d}x\mathrm{d}y.
\end{equation}
Let $\Psi=1$, from \eqref{adjoint}, we have $\psi=1$, then the equation \eqref{delta_sl} can be written as
\begin{equation}\label{delta_sl2}
	\int_{K_j}\left|u^{n+1}-\overline{u}^{n+1}\right| \mathrm{d}x\mathrm{d}y=\int_{\Delta K_j^\star}u^n\mathrm{d}x\mathrm{d}y.
\end{equation}
Since the analytical solution $u(t^{n})$ at $t^n$ time level is bounded and by the assumption, we have $\left\lVert u^{n}(x,y)-u(x,y,t^{n}) \right\rVert_{L^2}=O(h^3)$, therefore, $u^n(x,y)$ is bounded, we let $C_2=\tilde{C}_2\max\limits_{(x,y)\in \Delta K_j^{\star}}|u^n|$, then we have
$$
\frac{1}{h^2}\int_{K_j}\left|u^{n+1}-\overline{u}^{n+1}\right|\mathrm{d}\bm{\theta}=\frac{1}{h^2}\int_{\Delta K_j^\star}u^n\mathrm{d}\bm{\theta}\leqslant\frac{\max\limits_{(x,y)\in \Delta K_j^{\star}}|u^n|}{h^2}\int_{\Delta K_j^\star}1\mathrm{d}\bm{\theta}=\frac{C_2h^3}{h^2}.
$$
Therefore, we have 
\begin{equation}
	\left\lVert u^{n+1}-\overline{u}^{n+1} \right\rVert_{L^1}\leqslant C_2h^3.
\end{equation}
Then we let $\Psi=\left|u^{n+1}-\overline{u}^{n+1}\right|$, using the similar manner, we can obtain $$\left\lVert u^{n+1}-\overline{u}^{n+1} \right\rVert_{L^2}\leqslant C_2\sqrt{h^6}=C_2h^3.$$
In conclusion, $\left\lVert u^{n+1}(x,y)-u(x,y,t^{n+1}) \right\rVert_{L^2}\leqslant(C_1+C_2)h^3=O(h^3)$.

\section{Finding intersections between arcs}\label{appendixD}\quad
For two arcs with parametric equations:
\begin{equation}\label{parametric_arcs}
	\left\{ 
	\begin{array}{l}
		x_1(\xi)=a_1^x\xi^2+b_1^x\xi+c_1^x,\\  
		y_1(\xi)=a_1^y\xi^2+b_1^y\xi+c_1^y,
	\end{array}  
	\right.\text{ and }\left\{  
	\begin{array}{l}
		x_2(\eta)=a_2^x\eta^2+b_2^x\eta+c_2^x,\\  
		y_2(\eta)=a_2^y\eta^2+b_2^y\eta+c_2^y,
	\end{array}  
	\right.
\end{equation}
where $\xi,\eta\in\left[-1,1\right]$. Intersection points satisfy $\left\{  
\begin{array}{l}
	x_1(\xi)=x_2(\eta)\\  
	y_1(\xi)=y_2(\eta)
\end{array}  
\right.$, this system consists of two quadratic equations, by applying this, we can simplify the system to a single quartic equation:
$$\begin{array}{rl}
	a_1^x\xi^2+b_1^x\xi&=a_2^x\eta^2+b_2^x\eta+c_2^x-c_1^x:=\alpha,\\
	a_1^y\xi^2+b_1^y\xi&=a_2^y\eta^2+b_2^y\eta+c_2^y-c_1^y:=\beta.
\end{array}$$
We treat this system as  $\left[\begin{matrix}
	\xi^2 \\
	\xi \\
\end{matrix} \right]=\left[\begin{matrix}
	a_1^x & b_1^x\\
	a_1^y & b_1^y\\
\end{matrix} \right]^{-1}\left[\begin{matrix}
	\alpha \\
	\beta \\
\end{matrix} \right].$
By organizing coefficients, we have:
$$\begin{array}{rl}
	\xi^2&=A\eta^2+B\eta+C,\\
	\xi&=a\eta^2+b\eta+c.
\end{array}$$
This system yields a quartic equation in $\eta$ that
$$A\eta^2+B\eta+C=(a\eta^2+b\eta+c)^2.$$
So we find the intersection points of arcs by solving the equation.

\section{Determine if a vertex is inside a parabolic segment}\label{appendixE}\quad
We need to define the direction of the domain. A positive area is obtained from the contour integral under the assumption that the boundary of it has a counter-clockwise orientation.\\
\begin{figure}[ht]
	\centering
	\includegraphics*[width=0.32\linewidth]{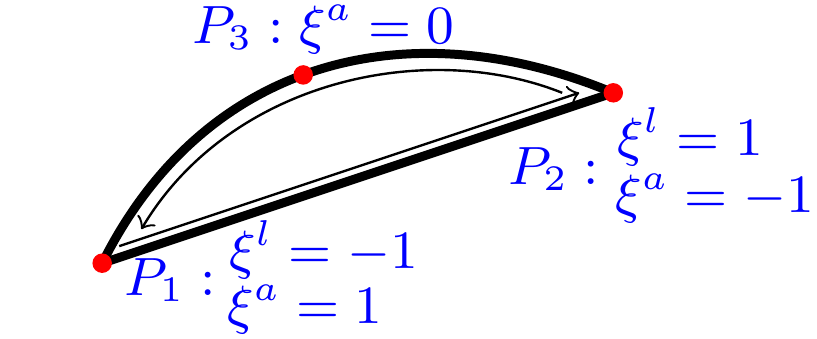}
	\caption{The parameter definition of a parabolic segment and along the direction of the arrow, which is the induced orientation, the parameters change from -1 to 1.}
	\label{parabolic_segment}
\end{figure}%The boundary comprises two distinct portions, see Fig.\ref{parabolic_segment}: a linear segment designated as base and an arc with $\xi^a$ as a parameter for quadratic polynomial. Three points $P_1,P_2,P_3$ in the physical space are sufficient to uniquely identify the arc. The points $P_1$, $P_2$ serve as termination points of the base, with the parabolic segment maintaining a leftward position during the transition from $P_1$ to $P_2$.\\
We choose the convention, showing in Fig.\ref{parabolic_segment} for parameterizing the arc and base to ensure the parameters vary from $-1$ to $1$ along the induced orientations of the base and the arc, respectively.\\
%It is a significant attribute of an arc which is utterly contained in coincident half-plane relative about the base.\\
Having the parametric description of parabolic segments, we can list the scenarios where the vertices are inside and outside a general parabolic segment $S$ in Fig.\ref{intersection_point}.
\begin{figure}[ht]
	\centering
	\includegraphics*[width=0.94\linewidth]{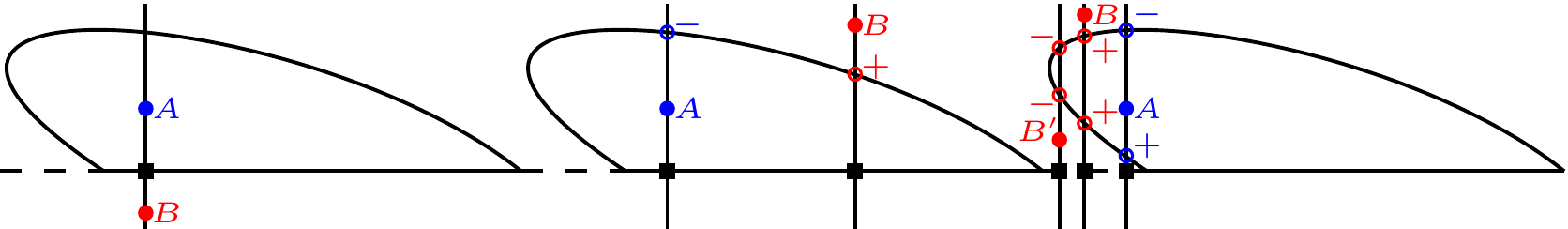}
	\caption{The method to determine if a point is inside the parabolic segment, the red point $B$ is outside the parabolic segment and the blue point $A$ is inside the parabolic segment in all panels, the square points represent the intersection between base of the segment and vertical line, along with hollow circles represent the intersection between curved edge of the segment and vertical line.}
	\label{intersection_point}
\end{figure}
\begin{enumerate}
	\item \textbf{Left panel:} The initial examination involves determining whether point is situated in the same half-plane as point $P_3$ of the arc associated with $S$.
	\item \textbf{Middle panel:} Next, we project the object onto the line that is established by the base of $S$. If the projection belongs to the base. We parameterize the vertical line by setting the parameter to zero at the point, and to a positive value at the projection. Afterwards, at $A$, the hollow-circle marked intersection of the arc and vertical line maps towards a negative parameter, while $B$ corresponding to a positive parameter.
	\item \textbf{Right panel:} If the projection is located beyond the base. For point $A$ located within $S$, the intersections are situated on opposing sides of $A$, which is indicative of a negative product of the parameters. In the case of $B$ and $B'$ that are external to $S$ yield positive products of the parameters.
\end{enumerate}

\section{Clipping algorithm}\label{appendixF}\quad
\begin{algorithm}[ht]
	\captionof{algorithm}{Clipping Between Convex Partition}\label{alg:clipping}
	\begin{algorithmic}[1]
		\Statex \textbullet~\textbf{given:} Convex partition P and Q.
		\Statex \textbullet~\textbf{requirement 1:} Find the vertices between the convex domains of Lagrangian and Eulerian element.
		\Statex \textbullet~\textbf{requirement 2:} Sort the vertices of intersection counter-clockwise.
		\Statex \textbullet~\textbf{returns:} A set of reordered vertices.
		\For {edge $P_i\in P$ and $Q_j\in Q$}
		\State Organize the coefficient of polynomial as in \ref{vertexes}(1).
		\State $\xi\gets$ roots of the polynomial.
		\If{$\xi \in [0,1]$}
		\State $\eta\gets$ any of the equation \eqref{parametric_arcs}
		\If{$\eta \in [0,1]$}
		\State The point is a vertex $I_i$ of intersection.
		\EndIf
		\EndIf
		\EndFor
		\For {vertex $p_i\in P$}
		\If{$p_i$ and the arc lie on the same half-plane}
		\State $V\gets$ perpendicular line through $p_i$.
		\State $\xi_{1,2}\gets$ intersection points of $V$ and arc.
		\If{$\xi_1$ and $\xi_2$ on the arc}
		\If{$\xi_1\cdot\xi_2<0$}
		$p_i$ is a vertex $I_i$ of intersection.
		\EndIf
		\Else
		\If{$\xi_1<0$}
		$p_i$ is a vertex $I_i$ of intersection.
		\EndIf
		\EndIf
		\EndIf
		\EndFor
		\State $O\gets \sum_{i=1}^N I_i/N$, $\bm{v}_0\gets$ any direction, $\bm{v}_i\gets \mathop{OI_i}\limits^{\rightarrow}$.
		\State $\theta_i\gets\arccos\left(\bm{v}_0\cdot\bm{v}_i/|\bm{v}_0||\bm{v}_i|\right)$, sort \{$I_i$\} by $\theta_i$.
	\end{algorithmic}
\end{algorithm}

\bibliographystyle{siamplain}
\bibliography{ref.bib}